\documentclass[12pt]{article}
\usepackage{amsmath, amsthm, amssymb, hyperref, color}
\usepackage{mathrsfs}
\usepackage[shortlabels]{enumitem}
\usepackage{graphicx}
\usepackage{xcolor}
\usepackage{ifthen}
\usepackage{xparse}

\definecolor{light-gray}{RGB}{240,240,240}
\oddsidemargin \evensidemargin
\newtheorem*{unnumberedtheorem}{Theorem}
\newtheorem{theorem}{Theorem}
\numberwithin{theorem}{section}
\newtheorem{proposition}[theorem]{Proposition}
\newtheorem{lemma}[theorem]{Lemma}
\newtheorem{corollary}[theorem]{Corollary}

\theoremstyle{definition}
\newtheorem{definition}[theorem]{Definition}

	\newcommand{\sizedescriptor}[2]
	{
		\ifthenelse{\equal{#1}{0}}{}{
		\ifthenelse{\equal{#1}{1}}{\big}{
		\ifthenelse{\equal{#1}{2}}{\Big}{
		\ifthenelse{\equal{#1}{3}}{\bigg}{
		\ifthenelse{\equal{#1}{4}}{\Bigg}{
		#2}}}}}
	}
	
	\newcommand{\proven}[1]{\underline{#1}\vspace{0.2em}\\}

	\newcommand{\df}[1]{{\bf{#1}}}

	\newcommand{\ke}{\mathrel{\simeq}}
	\newcommand{\dfeq}{\mathrel{\mathop:}=}

	\newcommand{\impl}{\Rightarrow}

	\newcommand{\rstr}[1]{\left.{#1}\right|}

	\newcommand{\parto}{\mathrel{\rightharpoonup}}
	
	\newcommand{\supp}[1]{\mathrm{supp}({#1})}

\newcommand{\all}[1]{\forall\hspace{0.2ex}{#1}\hspace{0.2ex}{.}\hspace{0.9ex}}
\newcommand{\some}[1]{\exists\hspace{0.2ex}{#1}\hspace{0.2ex}{.}\hspace{0.9ex}}

	\NewDocumentCommand{\set}
		{O{auto} m G{\empty}}
		{\sizedescriptor{#1}{\left}\{ {#2} \ifthenelse{\equal{#3}{}}{}{ \; \sizedescriptor{#1}{\middle}| \; {#3}} \sizedescriptor{#1}{\right}\}}

	\newcommand{\NN}{\mathbb{N}}
	\newcommand{\ZZ}{\mathbb{Z}}
	
	\newcommand{\RR}{\mathbb{R}}
	\newcommand{\intoo}[3][\RR]{{#1}_{(#2, #3)}}
	\newcommand{\intcc}[3][\RR]{{#1}_{[#2, #3]}}
	\newcommand{\intoc}[3][\RR]{{#1}_{(#2, #3]}}
	\newcommand{\intco}[3][\RR]{{#1}_{[#2, #3)}}

\newcommand{\ad}{n}
\newcommand{\md}{m}
\newcommand{\dst}{d}
\NewDocumentCommand{\rch}{G{}}{\tau_{#1}}
\newcommand{\pp}{p}
\newcommand{\hd}{\varkappa}
\newcommand{\lppb}{\lambda}

\newcommand{\pr}{pr}
\newcommand{\prv}{prv}

\NewDocumentCommand{\fl}{G{\smpp}}{f_{#1}}
\newcommand{\pf}{v}

\newcommand{\pt}[1]{\mathtt{#1}}
\newcommand{\smpp}{\pt{S}}
\newcommand{\gp}{\pt{X}}
\newcommand{\sgp}{\pt{Y}}
\newcommand{\cab}{\pt{C}}

\newcommand{\st}[1]{\mathcal{#1}}
\newcommand{\mnf}{\st{M}}
\newcommand{\gmnf}{\st{N}}
\newcommand{\otnm}[1]{\mnf_{#1}}
\newcommand{\ctnm}[1]{\overline{\mnf}_{#1}}
\newcommand{\smp}{\st{S}}
\NewDocumentCommand{\ma}{G{}}{\st{A}_{#1}}

\NewDocumentCommand{\ob}{G{\empty} G{\empty}}{\snbh{\mathbf{B}}{#2}{#1}}
\NewDocumentCommand{\cb}{G{\empty} G{\empty}}{\snbh{\overline{\mathbf{B}}}{#2}{#1}}
\newcommand{\fgr}{\mathbf{E}}
\NewDocumentCommand{\C}{G{}}{\mathcal{C}^{#1}}
\NewDocumentCommand{\snbh}{m m m}{{#1}_{#2}\ifthenelse{\equal{#3}{}}{}{\!\left(#3\right)}}
\NewDocumentCommand{\of}{G{\smpp} G{\pp}}{\snbh{\fgr}{#2}{#1}}
\NewDocumentCommand{\cf}{G{\smpp} G{\pp}}{\snbh{\overline{\fgr}}{#2}{#1}}
\NewDocumentCommand{\on}{G{\pp}}{\snbh{\fgr}{#1}{}}
\NewDocumentCommand{\cn}{G{\pp}}{\snbh{\overline{\fgr}}{#1}{}}
\NewDocumentCommand{\ol}{G{\smpp} G{\pp}}{\snbh{\mathbf{L}}{#2}{#1}}
\NewDocumentCommand{\cl}{G{\smpp} G{\pp}}{\snbh{\overline{\mathbf{L}}}{#2}{#1}}

\NewDocumentCommand{\bd}{G{\empty}}{\partial{#1}}
\NewDocumentCommand{\ts}{O{\mnf} G{\smpp}}{T_{#2}\hspace{0.1ex}{#1}}
\NewDocumentCommand{\ns}{O{\mnf} G{\smpp}}{N_{#2}\hspace{0.1ex}{#1}}
\NewDocumentCommand{\ls}{G{\gp}}{\Gamma\hspace{-0.2ex}_{#1}}
\newcommand{\vf}{V}
\newcommand{\avf}{\widetilde{\vf}}
\NewDocumentCommand{\lavf}{G{\smpp}}{\avf_{#1}}
\newcommand{\pu}[1][\smpp]{{f_{#1}}}

\newcommand{\pun}{f_P}
\newcommand{\flow}{\Phi}
\newcommand{\fd}{\st{D}}
\newcommand{\ph}{\text{---}}  

\NewDocumentCommand{\sld}{G{\smpp}}{\st{I}_{#1}}
\newcommand{\spr}[2]{\langle{#1}, {#2}\rangle}
\newcommand{\pd}[2]{\frac{\partial{#1}}{\partial{#2}}}

\newcommand{\pep}[1][\smpp]{q_{#1}}

\newcommand{\nhl}{\mathrm{hl}_{\smpp}{\gp}}
\newcommand{\pa}{\alpha_{\smpp}{\gp}}
\newcommand{\dr}{R}

\newcommand{\cs}{\mathscr{C}}
\newcommand{\ecs}{\widetilde{\cs}}

\newcommand{\lc}{L}
\newcommand{\hs}[1][\gp]{{\st{H}_{\smpp}{#1}}}

\newcommand{\arccot}{\mathrm{arccot}}

\newcommand{\se}[1][\smpp]{\st{A}_{#1}}
\newcommand{\de}[1][\smpp]{\st{B}_{#1}}
\newcommand{\tse}[1][\smpp]{\st{\widehat{A}}_{#1}}
\newcommand{\tde}[1][\smpp]{\st{\widehat{B}}_{#1}}
\newcommand{\sea}{\st{V}}
\newcommand{\npa}{\st{W}}

\title{\textbf{Finding the Homology of Manifolds using Ellipsoids}}
\author{Sara~Kali\v{s}nik and Davorin~Le\v{s}nik}
\date{}

\begin{document}
	
	\maketitle
	
	\begin{abstract}
		A standard problem in applied topology is how to discover topological invariants of data from a noisy point cloud that approximates it. We consider the case where a sample is drawn from a properly embedded $\C{1}$-submanifold without boundary in a Euclidean space. We show that we can deformation retract the union of ellipsoids, centered at sample points and stretching in the tangent directions, to the manifold. Hence the homotopy type, and therefore also the homology type, of the manifold is the same as that of the nerve complex of the cover by ellipsoids. By thickening sample points to ellipsoids rather than balls, our results require a smaller sample density than comparable results in the literature. They also advocate using elongated shapes in the construction of barcodes in persistent homology.
	\end{abstract}

	\section{Introduction}
		
		Data is often unstructured and comes in the form of a non-empty finite metric space, called a point cloud.  It is often very high dimensional even though data points are actually samples from a low-dimensional object (such as a manifold) that is embedded in a high-dimensional space.  One reason may be that many features are all measurements of the same underlying cause and therefore closely related to each other. For example, if you take photos of a single object from multiple angles simultaneously there is a lot overlap in the information captured by all those cameras. One of the main tasks of `manifold learning' is to design algorithms to estimate geometric and topological properties of the manifold from the sample points lying on this unknown manifold. 
		
		One successful framework for dealing with the problem of reconstructing shapes from point clouds is based on the notion of $\epsilon$-sample introduced by Amenta et al~\cite{amenta1999surface}. A sampling of a shape $\mnf$  is an $\epsilon$-sampling if every point~$\pt{P}$ in $\mnf$  has a sample point at distance at most $\epsilon \cdot \textrm{lfs}_\mnf(\pt{P})$, where $\textrm{lfs}_\mnf(\pt{P})$ is the local feature size of~$\pt{P}$, i.e.~the distance from~$\pt{P}$ to the medial axis of~$\mnf$. Surfaces smoothly embedded in~$\RR^3$ can be reconstructed homeomorphically from any $0.06$-sampling using the Cocone algorithm~\cite{doi:10.1142/S0218195902000773}. 
		
		One simple method for shape reconstructing is to output an offset of the sampling for a suitable value~$\alpha$ of the offset parameter. Topologically, this is equivalent to taking the \v{C}ech complex or the $\alpha$-complex~\cite{10.1145/174462.156635}. This leads to the problem of finding theoretical guarantees as to when an offset of a sampling has the same homotopy type as the underlying set. In other words, we need to find conditions on a point cloud~$\smp$ of a shape~$\mnf$ so that the thickening of~$\smp$ is homotopy equivalent to~$\mnf$. This only works if the point cloud is sufficiently close to~$\mnf$, i.e.~when there is a bound on the Hausdorff distance between $\smp$ and~$\mnf$.
		
		Niyogi, Smale and Weinberger~\cite{NSW} proved that this method indeed provides reconstructions having the correct homotopy type for densely enough sampled smooth submanifolds of $\RR^n$. More precisely, one can capture the homotopy type of a Riemannian submanifold~$\mnf$ without boundary of reach~$\rch$ in a Euclidean space from a finite $\frac{\epsilon}{2}$-dense sample $\smp \subseteq \mnf$ whenever $\epsilon < \sqrt{\frac{3}{5}} \rch$ by showing that the union of $\epsilon$-balls with centers in sample points deformation retracts to~$\mnf$.
		
		Let us denote the Hausdorff distance between $\smp$ and~$\mnf$ by~$\hd$ --- that is, every point in~$\mnf$ has an at most $\hd$-distant point in~$\smp$. We can rephrase the above result as follows: whenever $2\hd < \sqrt{\frac{3}{5}} \rch$, the homotopy type of~$\mnf$ is captured by a union of $\epsilon$-balls with centers in~$\smp$ for every $\epsilon \in \intoo{2\hd}{\sqrt{\frac{3}{5}} \rch}$. Thus the bound of the ratio
		\[\frac{\hd}{\rch} < \tfrac{1}{2} \sqrt{\tfrac{3}{5}} \approx 0.387\]
		represents how dense we need the sample to be in order to be able to recover the homotopy type of~$\mnf$.
		
		Other authors gave variants of Niyogi, Smale and Weinberger's result. In~\cite[Theorem~2.8]{10.1145/3275242}, the authors relax the conditions on the set we wish to approximate (it need not be a manifold, just any non-empty compact subset of a Euclidean space) and the sample (it need not be finite, just non-empty compact), but the price they pay for this is a lot lower upper bound on~$\frac{\hd}{\rch}$, which in their case is~$\frac{1}{6} \approx 0.167$. One can potentially improve the result by using local quantities ($\mu$-reach etc.) \cite{Chazal2009}, \cite{chazal2005lambda}, \cite{chazal2007stability}, \cite{Turner}, \cite{attali2010reconstructing}, \cite{attali2013vietoris} instead of the global reach~$\rch$, at least in situations when these are large compared to~$\rch$. Due to the difficulty and length of our current work, we take only global features into account, and leave the generalization to local features for future work.
		
		In practice producing a sufficiently dense sample can be difficult or requires a long time~\cite{dufresne2019sampling}, so relaxing the upper bound of~$\frac{\hd}{\rch}$ is desirable. The purpose of this paper is to prove that we can indeed relax this bound when sampling manifolds (though we allow a more general class than~\cite{NSW}) if we thicken sample points to \emph{ellipsoids} rather than balls. The idea is that since a differentiable manifold is locally well approximated by its tangent space, an ellipsoid with its major semi-axes in the tangent directions well approximates the manifold. This idea first appeared in~\cite{breiding2018learning}, where the authors construct a filtration of ``ellipsoid-driven complexes'', where the user can choose the ratio between the major (tangent) and the minor (normal) semi-axes. Their experiments showed that computing barcodes from ellipsoid-driven complexes strengthened the topological signal, in the sense that the bars corresponding to features of the data were longer. In our paper we make the ratio dependent on the persistence parameter and give a proof that the union of ellipsoids around sample points (under suitable assumptions) deformation retracts onto the manifold. Hence our paper gives theoretical guarantees that the union of ellipsoids captures the manifold's homotopy type, and thus further justifies the use of ellipsoid-inspired shapes to construct barcodes.
		
		In this paper we assume that the information about the reach of the manifold and its tangent and normal spaces in the sample points are given. In practice, these quantities can be estimated from the sample, see e.g.~\cite{aamari2019estimating, berenfeld2020estimating, zhang2004principal, kaslovsky2011optimal, zhang2011improved}.
		
		The central theorem of this paper (Theorem~\ref{theorem:main}) is the following:
		\begin{unnumberedtheorem}
			Let $\ad \in \NN$ and let $\mnf$ be a non-empty properly embedded $\mathcal{C}^1$-submanifold of~$\RR^\ad$ without boundary. Let $\smp \subseteq \mnf$ be a subset of~$\mnf$, locally finite in~$\RR^\ad$ (the sample from the manifold~$\mnf$). Let $\rch$ be the reach of~$\mnf$ in~$\RR^\ad$ and $\hd$ the Hausdorff distance between $\smp$ and~$\mnf$. Then for all $\pp \in \intcc{0.5\rch}{0.96\rch}$ which satisfy
			\[\hd < \sqrt{2\pp \big(\sqrt{\rch (\pp + 2\rch)} - \rch\big) - 0.55\rch^2}\]
			there exists a strong deformation retraction from~$\on$ (the union of open ellipsoids around sample points with normal semi-axes of length~$\pp$ and tangent semi-axes of length~$\sqrt{\rch \pp + \pp^2}$) to~$\mnf$. In particular, $\mnf$, $\on$ and the nerve complex of the ellipsoid cover $(\of)_{\smpp \in \smp}$ are homotopy equivalent, and so have the same homology.
		\end{unnumberedtheorem}
		
		By replacing the balls with ellipsoids, we manage to push the upper bound on~$\frac{\hd}{\rch}$ to approximately~$0.913$, an improvement by a factor of about~$2.36$ compared to~\cite{NSW}. In other words, our method allows samples with less than half the density.
		
		The paper is organized as follows. Section~\ref{section:general-definitions} lays the groundwork for the paper, providing requisite definitions and deriving some results for general differentiable submanifolds of Euclidean spaces. In Section~\ref{section:bounds-on-persistence-parameter} we calculate theoretical bounds on the persistence parameter~$\pp$: the lower bound ensures that the union of ellipsoids covers the manifold and the upper bound ensures that the union does not intersect the medial axis. Part of our proof relies on the normal deformation retraction working on intersections of ellipsoids which appears too difficult to prove theoretically by hand, so we resort to a computer program, explained in Section~\ref{section:program}. In Section~\ref{section:deformation-retraction-construction} we construct a deformation retraction from the union of ellipsoids to the manifold. The section is divided into several subsections for easier reading. Section~\ref{section:main-result} collects the results from the paper to prove the main theorem. In Section~\ref{section:discussion} we discuss our results and future work.
		
		\subsection*{Notation}
		
			Natural numbers~$\NN = \set{0, 1, 2,\ldots}$ include zero. Unbounded real intervals are denoted by~$\RR_{> a}$, $\RR_{\leq a}$ etc. Bounded real intervals are denoted by~$\intoo{a}{b}$ (open), $\intcc{a}{b}$ (closed) etc.
			
			\begin{tabular}{rl}
				Glossary: &\\
				&\\
				$\dst$ & Euclidean distance in~$\RR^\ad$ \\
				$\gmnf$ & a submanifold of~$\RR^\ad$ \\
				$\mnf$ & $\md$-dimensional $\C{1}$-submanifold of~$\RR^\ad$, embedded as a closed subset \\
				$\otnm{r}$ & open $r$-offset of~$\mnf$, i.e.~$\set{\gp \in \RR^\ad}{\dst(\mnf, \gp) < r}$ \\
				$\ctnm{r}$ & closed $r$-offset of~$\mnf$, i.e.~$\set{\gp \in \RR^\ad}{\dst(\mnf, \gp) \leq r}$ \\
				$\ts{\gp}$ & tangent space on~$\mnf$ at~$\gp$ \\
				$\ns{\gp}$ & normal space on~$\mnf$ at~$\gp$ \\
				$\smp$ & manifold sample (a subset of~$\mnf$), non-empty and locally finite \\
				$\hd$ & the Hausdorff distance between $\mnf$ and~$\smp$ \\
				$\ma$ & the medial axis of~$\mnf$ \\
				$\rch$ & the reach of~$\mnf$ \\
				$\pp$ & persistence parameter \\
				$\of$ & open ellipsoid with the center in a sample point~$\smpp \in \smp$ \\
				& with the major semi-axes tangent to~$\mnf$ \\
				$\cf$ & closed ellipsoid with the center in a sample point~$\smpp \in \smp$ \\
				& with the major semi-axes tangent to~$\mnf$ \\
				$\bd{\cf}$ & the boundary of~$\cf$, i.e.~$\cf \setminus \of$ \\
				$\on$ & the union of open ellipsoids over the sample, $\bigcup_{\smpp \in \smp} \of$ \\
				$\cn$ & the union of closed ellipsoids over the sample, $\bigcup_{\smpp \in \smp} \cf$ \\
				$\pr$ & the map $\ma^\complement \to \mnf$ taking a point to the unique closest point on~$\mnf$ \\
				$\prv$ & the map taking a point~$\gp$ to the vector $\pr(\gp) - \gp$ \\
				$\lavf$ & auxiliary vector field, defined on~$\of$ \\
				$\avf$ & auxiliary vector field, defined on~$\on$ \\
				$\vf$ & the vector field of directions for the deformation retraction \\
				$\flow$ & the flow of the vector field~$\vf$ \\
				$\dr$ & a deformation retraction from $\on$ to a tubular neighbourhood of~$\mnf$
			\end{tabular}
	
	\section{General Definitions}\label{section:general-definitions}
	
		All constructions in this paper are done in an ambient Euclidean space~$\RR^\ad$, $\ad \in \NN$, equipped with the usual Euclidean metric~$\dst$. We will use the symbol~$\gmnf$ for a general submanifold of~$\RR^\ad$.
		
		By definition each point~$\gp$ of a manifold $\gmnf$ has a neighbourhood, homeomorphic to a Euclidean space or a closed Euclidean half-space. The dimension of this (half-)space is the \df{dimension of~$\gmnf$ at~$\gp$}. Different points of a manifold can have different dimensions\footnote{A simple example is the submanifold of the plane, given by the equation $(x^2 + y^2)^2 = x^2 + y^2$ (a union of a point and a circle).}, though the dimension is constant on each connected component. In this paper, when we say that $\gmnf$ is an \df{$\md$-dimensional manifold}, we mean that it has dimension~$\md$ at \emph{every} point.
		
		We quickly recall from general topology that it is equivalent for a subset of a Euclidean space to be closed and to be properly embedded.
		
		\begin{proposition}\label{proposition:properly-embedded-submanifolds}
			Let $(\st{X}, d)$ be a metric space in which every closed ball is compact (every Euclidean space~$\RR^\ad$ satisfies this property). The following statements are equivalent for any subset $\st{S} \subseteq \st{X}$. 
			\begin{enumerate}
				\item
					$\st{S}$ is a closed subset of~$\st{X}$.
				\item
					$\st{S}$ is properly embedded into~$\st{X}$, i.e.~the inclusion $\st{S} \hookrightarrow \st{X}$ is a proper map\footnote{Recall that a map is \df{proper} when the preimage of every compact subspace of the codomain is compact.}.
				\item
					$\st{S}$ is empty or distances from points in the ambient space to $\st{S}$ are attained. That is, for every $\gp \in \st{X}$ there exists $\sgp \in \st{S}$ such that $d(\gp, \st{S}) = d(\gp, \sgp)$.
			\end{enumerate}
		\end{proposition}
		
		\begin{proof}
			\
			\begin{itemize}
				\item\proven{$(1 \impl 2)$}
					If $\st{S}$ is closed in~$\st{X}$, then its intersection with a compact subset of~$\st{X}$ is compact, so $\st{S}$ is properly embedded into~$\st{X}$.
				\item\proven{$(2 \impl 3)$}
					If $\st{S}$ is non-empty, pick $\pt{S} \in \st{S}$. For any $\gp \in \st{X}$ we have $d(\gp, \st{S}) \leq d(\gp, \pt{S})$, so $d(\gp, \st{S}) = d\big(\gp, \st{S} \cap \cb{\gp}{d(\gp, \pt{S})}\big)$. Since $\st{S}$ is properly embedded in~$\st{X}$, its intersection with the compact closed ball $\cb{\gp}{d(\gp, \pt{S})}$ is compact also. A continuous map from a non-empty compact space into reals attains its minimum, so there exists $\sgp \in \st{S}$ such that $d(\gp, \sgp) = d\big(\gp, \st{S} \cap \cb{\gp}{d(\gp, \pt{S})}\big) = d(\gp, \st{S})$.
				\item\proven{$(3 \impl 1)$}
					The empty set is closed. Assume that $\st{S}$ is non-empty. Then for every point in the closure $\gp \in \overline{\st{S}}$ we have $d(\gp, \st{S}) = 0$. By assumption this distance is attained, i.e.~we have $\sgp \in \st{S}$ such that $d(\gp, \sgp) = 0$, so $\gp = \sgp \in \st{S}$. Thus $\overline{\st{S}} \subseteq \st{S}$, so $\st{S}$ is closed.
			\end{itemize}
		\end{proof}
		
		In this paper we consider exclusively submanifolds of a Euclidean space which are properly embedded, so closed subsets. We mostly use the term `properly embedded' instead of `closed' to avoid confusion: the term `closed manifold' is usually used in the sense `compact manifold with no boundary' which is a stronger condition (a properly embedded submanifold need not be compact or without boundary, though every compact submanifold is properly embedded).
		
		A manifold can have smooth structure up to any order $k \in \NN_{\geq 1} \cup \set{\infty}$; in that case it is called a $\C{k}$-manifold. A $\C{k}$-submanifold of~$\RR^\ad$ is a $\C{k}$-manifold which is a subset of~$\RR^\ad$ and the inclusion map is $\C{k}$.
		
		If $\gmnf$ is at least a $\C{1}$-manifold, one may abstractly define the tangent space~$\ts[\gmnf]{\gp}$ and the normal space~$\ns[\gmnf]{\gp}$ at any point $\gp \in \gmnf$ ($\gp$ is allowed to be a boundary point). As we restrict ourselves to submanifolds of~$\RR^\ad$, we also treat the tangent and the normal space as affine subspaces of~$\RR^n$, with the origins of $\ts[\gmnf]{\gp}$ and $\ns[\gmnf]{\gp}$ placed at~$\gp$. The dimension of $\ts[\gmnf]{\gp}$ (resp.~$\ns[\gmnf]{\gp}$) is the same as the dimension (resp.~codimension) of $\gmnf$ at~$\gp$. Because of this and because $\ts[\gmnf]{\gp}$ and $\ns[\gmnf]{\gp}$ are orthogonal, they together generate~$\RR^\ad$.
		
		\begin{definition}
			Let $\gmnf$ be a $\C{1}$-submanifold of~$\RR^\ad$, $\gp \in \gmnf$ and $\md$ the dimension of~$\gmnf$ at~$\gp$.
			\begin{itemize}
				\item
					A \df{tangent-normal coordinate system} at $\gp \in \gmnf$ is an $\ad$-dimensional orthonormal coordinate system with the origin in~$\gp$, the first $\md$ coordinate axes tangent to~$\gmnf$ at~$\gp$ and the last $\ad-\md$ axes normal to~$\gmnf$ at~$\gp$.
				\item
					A \df{planar tangent-normal coordinate system} at $\gp \in \gmnf$ is a two-dimensional plane in~$\RR^\ad$ containing~$\gp$, together with the choice of an orthonormal coordinate system lying on it, with the origin in~$\gp$, the first axis (the abscissa) tangent to~$\gmnf$ at~$\gp$ and the second axis (the ordinate) normal to~$\gmnf$ at~$\gp$.
			\end{itemize}
		\end{definition}
		
		Recall from Proposition~\ref{proposition:properly-embedded-submanifolds} that distances from points to a non-empty properly embedded submanifold are attained. However, these distances need not be attained in just one point. As usual, we define the \df{medial axis}~$\ma{\gmnf}$ of a submanifold $\gmnf \subseteq \RR^\ad$ as the set of all points in the ambient space for which the distance to~$\gmnf$ is attained in at least two points:
		\[\ma{\gmnf} \dfeq \set[1]{\gp \in \RR^\ad}{\some{\sgp', \sgp'' \in \gmnf}{\sgp' \neq \sgp'' \land \dst(\gp, \sgp') = \dst(\gp, \sgp'') = \dst(\gp, \gmnf)}}.\]
		If $\gmnf$ is empty, so is~$\ma{\gmnf}$, though the medial axis can be empty even for non-empty manifolds (consider for example a line or a line segment in a plane). The manifold and its medial axis are always disjoint.
		
		The \df{reach} of~$\gmnf$, denoted by $\rch{\gmnf}$, is the distance between the manifold $\gmnf$ and its medial axis~$\ma{\gmnf}$ (if $\ma{\gmnf}$ is empty, the reach is defined to be~$\infty$).
		
		\begin{definition}
			Let $\gmnf$ be a $\C{1}$-submanifold of~$\RR^\ad$, $\gp \in \gmnf$ and $\vec{N}$ a non-zero normal vector to~$\gmnf$ at~$\gp$. The $\rch{\gmnf}$-ball, \df{associated} to~$\gp$ and~$\vec{N}$, is the closed ball (in~$\RR^\ad$, so $\ad$-dimensional) with radius~$\rch{\gmnf}$ and centered at $\gp + \rch{\gmnf} \frac{\vec{N}}{\|\vec{N}\|}$, which therefore touches~$\gmnf$ at~$\gp$.\footnote{If $\rch{\gmnf} = \infty$, this ``ball'' is the whole closed half-space which contains~$\vec{N}$ and the boundary of which is the hyperplane, orthogonal to~$\vec{N}$, which contains~$\gp$.} A $\rch{\gmnf}$-ball, \df{associated} to~$\gp$, is the $\rch{\gmnf}$-ball, associated to~$\gp$ and some non-zero normal vector to~$\gmnf$ at~$\gp$.
		\end{definition}
		
		The significance of associated $\rch{\gmnf}$-balls is that they provide restrictions to where a manifold is situated. Specifically, a manifold is disjoint with the interior of its every associated $\rch{\gmnf}$-ball.
		
		We will approximate manifolds with a union of ellipsoids (similar as to how one uses a union of balls to approximate a subspace in the case of a \v{C}ech complex). The idea is to use ellipsoids which are elongated in directions, tangent to the manifold, so that they ``extend longer in the direction the manifold does'', so that we require a sample with lower density.
		
		Let us define the kind of ellipsoids we use in this paper.
		
		\begin{definition}
			Let $\gmnf$ be a $\C{1}$-submanifold of~$\RR^\ad$ and $\pp \in \RR_{> 0}$. The \df{tangent-normal open} (resp.~\df{closed}) \df{$\pp$-ellipsoid at $\gp \in \gmnf$} is the open (resp.~closed) ellipsoid in~$\RR^\ad$ with the center in~$\gp$, the tangent semi-axes of length $\sqrt{\rch{\gmnf} \pp + \pp^2}$ and the normal semi-axes of length~$\pp$. Explicitly, in a tangent-normal coordinate system at~$\gp$ the tangent-normal open and closed $\pp$-ellipsoids are given by
			\begin{align*}
				\of{\gp} &\dfeq \set[2]{(x_1, \ldots, x_\ad) \in \RR^\ad}{\frac{x_1^2 + \ldots + x_\md^2}{\rch{\gmnf} \pp + \pp^2} + \frac{x_{\md + 1}^2 + \ldots + x_\ad^2}{\pp^2} < 1}, \\
				\cf{\gp} &\dfeq \set[2]{(x_1, \ldots, x_\ad) \in \RR^\ad}{\frac{x_1^2 + \ldots + x_\md^2}{\rch{\gmnf} \pp + \pp^2} + \frac{x_{\md + 1}^2 + \ldots + x_\ad^2}{\pp^2} \leq 1},
			\end{align*}
			where $\md$ denotes the dimension of~$\gmnf$ at~$\gp$. If $\rch{\gmnf} = \infty$, then these ``ellipsoids'' are simply thickenings of~$\ts[\gmnf]{\gp}$:
			\begin{align*}
				\of{\gp} &\dfeq \set[1]{(x_1, \ldots, x_\ad) \in \RR^\ad}{\sqrt{x_{\md + 1}^2 + \ldots + x_\ad^2} < \pp}, \\
				\cf{\gp} &\dfeq \set[1]{(x_1, \ldots, x_\ad) \in \RR^\ad}{\sqrt{x_{\md + 1}^2 + \ldots + x_\ad^2} \leq \pp}.
			\end{align*}
			Observe that the definitions of ellipsoids are independent of the choice of the tangent-normal coordinate system; they depend only on the submanifold itself.
		\end{definition}
		
		The value~$\pp$ in the definition of ellipsoids serves as a ``persistence parameter''~\cite{ghrist}, \cite{topodata}, \cite{ZC}, \cite{elz-tps-02}, \cite{breiding2018learning}. We purposefully do not take ellipsoids which are similar at all~$\pp$ (which would mean that the ratio between the tangent and the normal semi-axes was constant). Rather, we want ellipsoids which are more elongated (have higher eccentricity) for smaller~$\pp$. This is because on a smaller scale a smooth manifold more closely aligns with its tangent space, and then so should the ellipsoids. We want the length of the major semi-axes to be a function of~$\pp$ with the following properties: for each $\pp$ its value is larger than~$\pp$, and when $\pp$ goes to~$0$, the function value also goes to~$0$, but the eccentricity goes to~$1$. In addition, the function should allow the following argument. If we change the unit length of the coordinate system, but otherwise leave the manifold ``the same'', we want the ellipsoids to remain ``the same'' as well, but the reach of the manifold changes by the same factor as the unit length, which the function should take into account. The simplest function satisfying all these properties is arguably $\sqrt{\rch{\gmnf} \pp + \pp^2}$, which turns out to work for the results we want.
		
		Figure~\ref{figure:tangent-normal-coordinate-system} shows an example, how a manifold, associated balls and a tangent-normal ellipsoid look like in a tangent-normal coordinate system at some point on the manifold.
		\begin{figure}[!ht]
			\centering
			\includegraphics[width=0.7\textwidth]{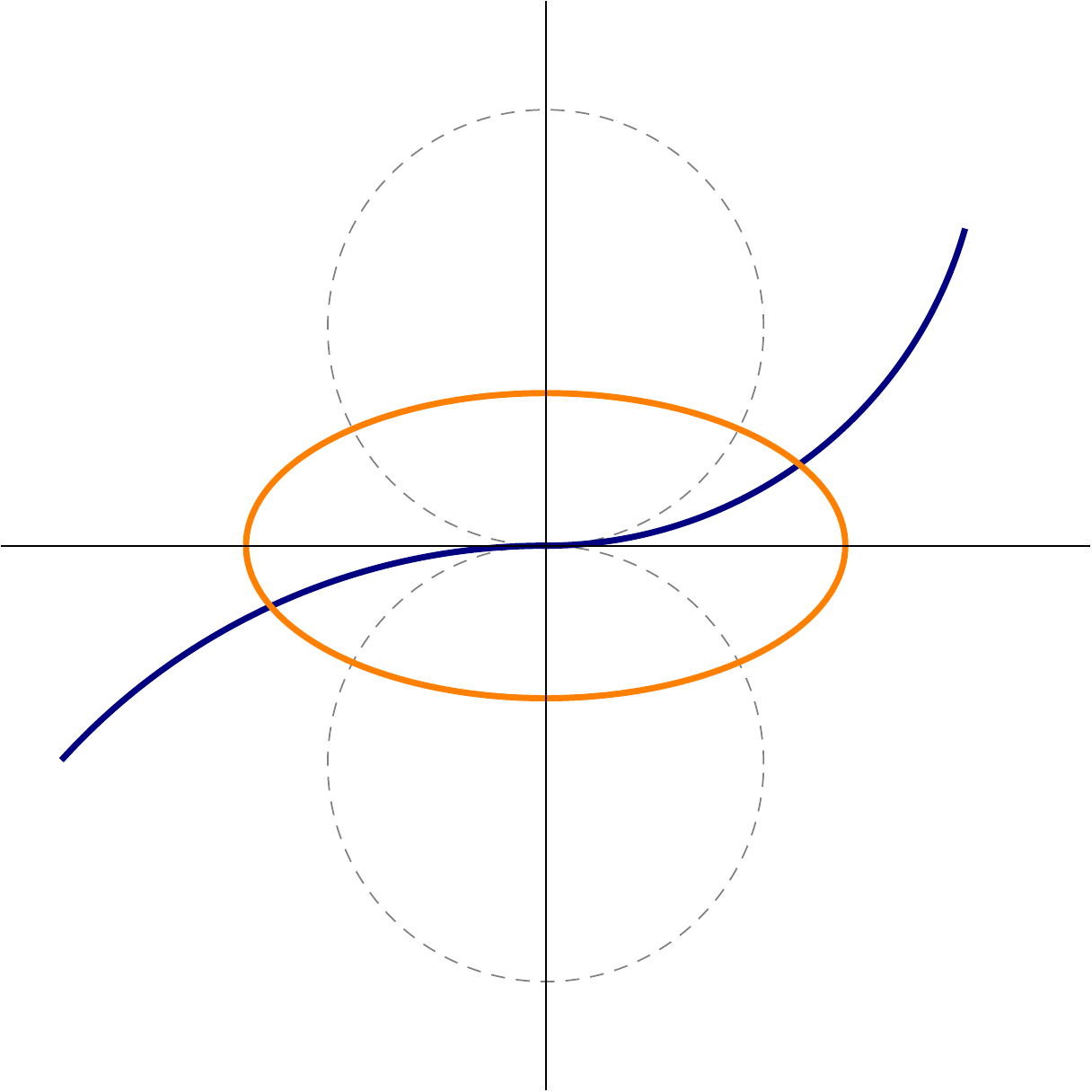}
			\caption{Tangent-normal coordinate system.}\label{figure:tangent-normal-coordinate-system}
		\end{figure}
		
		We now prove a few results that will be useful later.
		
		\begin{lemma}\label{lemma:general-manifold-properties}
			Let $\gmnf$ be a properly embedded $\C{1}$-submanifold of~$\RR^\ad$. Let $\gp \in \gmnf$ and let $\md$ be the dimension of~$\gmnf$ at~$\gp$. Assume $0 < \md < \ad$.
			\begin{enumerate}
				\item\label{lemma:general-manifold-properties:planar-tangent-normal-coordinate-system}
					For every $\sgp \in \RR^\ad$ a planar tangent-normal coordinate system at $\gp \in \gmnf$ exists which contains~$\sgp$. Without loss of generality we may require that the coordinates of~$\sgp$ in this coordinate system are non-negative ($\sgp$ lies in the closed first quadrant).
				\item
					If $\pp \in \RR_{> 0}$, $\sgp \in \bd{\cf{\gp}}$ and $\vec{N}$ is a vector, normal to $\bd{\cf{\gp}}$ at~$\sgp$, then we may additionally assume that the planar tangent-normal coordinate system from the previous item contains~$\vec{N}$.
				\item
					Let $\st{O}$ be a closed $(\ad-\md+1)$-dimensional ball, $\C{1}$-embedded in~$\RR^\ad$ (in particular $\bd{\st{O}}$ is a $\C{1}$-submanifold of~$\RR^\ad$, diffeomorphic to an $(\ad-\md)$-dimensional sphere). Assume that $\st{O} \cap \bd{\gmnf} = \emptyset$ and that $\gmnf$ and $\bd{\st{O}}$ intersect transversely in~$\gp$. Then $\gp$ is not the only intersection point, i.e.~there exists $\sgp \in \gmnf \cap \bd{\st{O}} \setminus \set{\gp}$.
				\item\label{lemma:general-manifold-properties:bound-on-distance-to-manifold}
					Assume $\rch{\gmnf} < \infty$. Let $\sgp \in \RR^n$ and let $(y_T, y_N)$ be the (non-negative) coordinates of~$\sgp$ in the planar tangent-normal coordinate system from the first item. Let $\st{D}$ be the set of centers of all $\rch{\gmnf}$-balls, associated to~$\gp$ (i.e.~the $(\ad-\md-1)$-dimensional sphere within~$\ns[\gmnf]{\gp}$ with the center in~$\gp$ and the radius~$\rch{\gmnf}$). Let $\st{C}$ be the cone which is the convex hull of~$\st{D} \cup \set{\sgp}$, and assume that $\st{C} \cap \bd{\gmnf} = \emptyset$. Then
					\[\dst(\gmnf, \sgp) \leq \sqrt{y_T^2 + (y_N + \rch{\gmnf})^2} - \rch{\gmnf}.\]
			\end{enumerate}
		\end{lemma}
		
		\begin{proof}
			\
			\begin{enumerate}
				\item
					Fix an $\ad$-dimensional tangent-normal coordinate system at $\gp \in \gmnf$, and let $(y_1, \ldots, y_\ad)$ be the coordinates of~$\sgp$. Let $\vec{a} = (y_1, \ldots, y_\md, 0, \ldots, 0)$, $\vec{b} = (0, \ldots, 0, y_{\md+1}, \ldots, y_\ad)$. If both $\vec{a}$ and $\vec{b}$ are non-zero, they define a (unique) planar tangent-normal coordinate system at~$\gp$ which contains~$\sgp$. If $\vec{a}$ is zero (resp.~$\vec{b}$ is zero), choose an arbitrary tangent (resp.~normal) direction (we may do this since $0 < \md < \ad$).
				\item
					Assume that $\sgp \in \bd{\cf{\gp}}$ and $\vec{N}$ is a direction, normal to~$\bd{\cf{\gp}}$. In the $\ad$-dimensional tangent-normal coordinate system from the previous item, the boundary $\bd{\cf}$ is given by the equation
					\[\sum_{i = 1}^{\md} \frac{x_i^2}{\rch{\gmnf} \pp + \pp^2} + \sum_{j = \md + 1}^{\ad} \frac{x_j^2}{\pp^2} = 1.\]
					The gradient of the left-hand side, up to a scalar factor, is
					\[\big(\tfrac{x_1}{\rch{\gmnf} \pp + \pp^2}, \ldots, \tfrac{x_\md}{\rch{\gmnf} \pp + \pp^2}, \tfrac{x_{\md + 1}}{\pp^2}, \ldots, \tfrac{x_\ad}{\pp^2}\big) = \tfrac{1}{\rch{\gmnf} \pp + \pp^2} \, \vec{a} + \tfrac{1}{\pp^2} \, \vec{b}.\]
					The vector~$\vec{N}$ has to be parallel to it since $\bd{\cf{\gp}}$ has codimension~$1$, i.e.~a non-zero $\lambda \in \RR$ exists such that $\vec{N} = \tfrac{\lambda}{\rch{\gmnf} \pp + \pp^2} \, \vec{a} + \tfrac{\lambda}{\pp^2} \, \vec{b}$. Hence $\vec{N}$ also lies in the plane, determined by $\vec{a}$ and~$\vec{b}$. This proof works for $\rch{\gmnf} < \infty$, but the required modification for $\rch{\gmnf} = \infty$ is trivial.
				\item
					Since $\st{O}$ is a compact $(\ad-\md+1)$-dimensional disk and $\bd{\gmnf}$ is closed, some thickening of~$\st{O}$ exists --- denote it by~$\st{T}$ --- which is diffeomorphic to an $\ad$-dimensional ball and is still disjoint with~$\bd{\gmnf}$. With a small perturbation of~$\gmnf$ around $\gmnf \cap \bd{\st{T}}$ (but away from the intersection $\gmnf \cap \st{O}$ which must remain unchanged) we can achieve that $\gmnf$ and $\bd{\st{T}}$ only have transversal intersections~\cite{lee2013smooth}.
					
					Imagine $\RR^\ad$ embedded into its one-point compactification~$S^\ad$ (denote the added point by~$\infty$) in such a way that $\st{T}$ is a hemisphere. Replace the part of~$\gmnf$ outside of~$\st{T}$ with a copy of $\gmnf \cap \st{T}$, reflected over~$\bd{\st{T}}$, and denote the obtained space by~$\gmnf'$. This is an embedding of the so-called \df{double} of the manifold~$\gmnf \cap \st{T}$. Then $\gmnf'$ is a manifold without boundary, closed in the sphere, and therefore compact. If necessary, perturb it slightly around the point~$\infty$, so that $\infty \notin \gmnf'$. Hence $\gmnf'$ is a compact submanifold in~$\RR^\ad$ without boundary and $\C{1}$-smooth everywhere except possibly on $\gmnf' \cap \bd{\st{T}}$. The double of a $\C{1}$-manifold can be equipped with a $\C{1}$-structure. Therefore we can use Whitney's approximation theorem~\cite{lee2013smooth} to adjust the embedding of $\gmnf'$ on a neighbourhood of $\bd{\st{T}}$ away from~$\st{O}$, so that it is $\C{1}$-smooth everywhere. The result is a compact manifold~$\gmnf'$ without boundary satisfying all the properties we required of~$\gmnf$, and we have $\gmnf' \cap \st{O} = \gmnf \cap \st{O}$. This shows that we may without loss of generality assume that $\gmnf$ is compact without boundary.
					
					Any compact $k$-dimensional submanifold of~$S^\ad$ without boundary represents an element in the cohomology $H^k(S^\ad; \ZZ_2)$ (we take the $\ZZ_2$-coeficients, so that we do not have to worry about orientation). For elements $[\gmnf] \in H^\md(S^\ad; \ZZ_2)$ and $[\bd{\st{O}}] \in H^{\ad-\md}(S^\ad; \ZZ_2)$ we know~\cite{bredon2013topology} that their cup-product $[\gmnf] \mathop{\smile} [\bd{\st{O}}] \in H^\ad(S^\ad; \ZZ_2)$ is the intersection number of $\gmnf$ and~$\bd{\st{O}}$ (times the generator). Since the cohomology of~$S^\ad$ is trivial except in dimensions~$0$ and~$\ad$, we have $[\gmnf] = [\bd{\st{O}}] = 0$, and hence $[\gmnf] \mathop{\smile} [\bd{\st{O}}] = 0$. But the local intersection number in the transversal intersection~$\gp$ is~$1$, and the intersection number is the sum of local ones, so $\gp$ cannot be the only point in $\gmnf \cap \bd{\st{O}}$.
				\item
					First consider the case when $\sgp \in \ns[\gmnf]{\gp}$, i.e.~$y_T = 0$. Then
					\[\dst(\gmnf, \sgp) \leq \dst(\gp, \sgp) = y_N = \sqrt{y_T^2 + (y_N + \rch{\gmnf})^2} - \rch{\gmnf}.\]
					Now suppose $\sgp \notin \ns[\gmnf]{\gp}$. Then the cone~$\st{C}$ is homeomorphic to an $(\ad-\md+1)$-dimensional closed ball. This~$\st{C}$ and its boundary are smooth everywhere except in~$\sgp$ and on~$\st{D}$. Let $\st{E}$ be the $(\ad-\md+1)$-dimensional affine subspace which contains~$\sgp$ and~$\ns[\gmnf]{\gp}$ (thus the whole~$\st{C}$). We can smooth~$\bd{\st{C}}$ around the centers of the associated balls within~$\st{E}$ without affecting the intersection with~$\gmnf$ since $\gmnf$ is disjoint with the interiors of the associated $\rch{\gmnf}$-balls. If $\sgp \in \gmnf$, then $\dst(\gmnf, \sgp) = 0 \leq \sqrt{y_T^2 + (y_N + \rch{\gmnf})^2} - \rch{\gmnf}$, and we are done. If $\sgp \notin \gmnf$, then $\dst(\gmnf, \sgp) > 0$ since $\gmnf$ is a closed subset. Then we can also smooth $\bd{\st{C}}$ around~$\sgp$ within~$\st{E}$ without affecting the intersection with~$\gmnf$. The boundary smoothed in this way is diffeomorphic to an $(\ad-\md)$-dimensional sphere, and so by the generalized Schoenflies theorem splits~$\st{E}$ into the inner part, diffeomorphic to an $(\ad-\md+1)$-dimensional ball, and the outer unbounded part. Since $\gmnf$ intersects~$\bd{\st{C}}$ and therefore also its smoothed version orthogonally in~$\gp$, this intersection is transversal. By the previous item another intersection point $\gp' \in \gmnf \cap \bd{\st{C}} \setminus \set{\gp}$ exists. It cannot lie in~$\ns[\gmnf]{\gp}$ since we would then have a manifold point in the interior of some associated ball, so $\gp'$ must lie on the lateral surface of the cone. That is, $\gp'$ lies on the line segment between $\sgp$ and some associated ball center, but it cannot lie in the interior of the associated ball, so $\dst(\gp', \sgp)$ is bounded by the distance between $\sgp$ and the furthest associated ball center, decreased by~$\rch{\gmnf}$. The furthest center is the one within the starting planar tangent-normal coordinate system that has coordinates~$(0, -\rch{\gmnf})$. Thus
					\[\dst(\gmnf, \sgp) \leq \dst(\gp', \sgp) \leq \sqrt{y_T^2 + (y_N + \rch{\gmnf})^2} - \rch{\gmnf}.\]
			\end{enumerate}
		\end{proof}
		
		\begin{lemma}\label{lemma:inside-ellipsoid}
			Let $A, B \in \RR_{\geq 0}$ which are not both~$0$ and let $\rch \in \RR_{> 0}$. Then a unique $q \in \RR_{> 0}$ exists which solves the equation
			\[\frac{A}{\rch q + q^2} + \frac{B}{q^2} = 1.\]
			Moreover, this $q$ depends continuously on $A$ and~$B$, and if $(A, B) \to (0, 0)$ (with $\rch$ fixed), then $q \to 0$.
		\end{lemma}
		
		\begin{proof}
			If $A = 0$, then clearly $q = \sqrt{B} > 0$ works. If $B = 0$, then the unique positive solution to the quadratic equation $q^2 + \rch q - A = 0$ is $q = \frac{\sqrt{\rch^2 + 4A} - \rch}{2}$.
			
			Assume that $A, B > 0$. Multiply the equation from the lemma by~$q^2 (\rch+q)$ and take all terms to one side of the equation to get
			\[q^3 + \rch q^2 - (A+B) q - \rch B = 0.\]
			Define the function $f\colon \RR \to \RR$ by $f(x) \dfeq x^3 + \rch x^2 - (A+B) x - \rch B$. The zeros of its derivative $f'(x) = 3x^2 + 2\rch x - (A+B)$ are
			\[\frac{-\rch \pm \sqrt{\rch^2 + 3(A+B)}}{3};\]
			since $A+B > 0$, both zeros are real and one is negative, the other positive. Let $z$ denote the positive zero. We have $f(0) = -\rch B < 0$ and $f'$ is $\leq 0$ on~$\intcc{0}{z}$, so $f$ cannot have a zero here, and $f(z) < 0$. Since $f$ is strictly increasing on $\RR_{> z}$ and $\lim_{x \to \infty} f(x) = \infty$, we conclude that $f$ has a unique zero on~$\RR_{> z}$ and therefore also on~$\RR_{> 0}$.
			
			Since $q$ is the root of the polynomial $q^3 + \rch q^2 - (A+B) q - \rch B$ and polynomial roots depend continuously on the coefficients, $q$ depends continuously on $A$ and~$B$ as well. In particular, if $A$ and $B$ tend to~$0$, then $q$ tends to one of the roots of $q^3 + \rch q^2$. It cannot tend to~$-\rch$ since it is positive, so it tends to~$0$.
		\end{proof}
		
		Given a properly embedded $\C{1}$-submanifold $\gmnf \subseteq \RR^\ad$ without boundary and a point $\sgp \in \gmnf$, the dimension of~$\gmnf$ at which we denote by~$\md$, let us define the continuous function $\pep[\sgp]\colon \RR^\ad \to \RR_{\geq 0}$ in the following way.
		\begin{definition}\label{definition:depth-parameter}
			If $\rch{\gmnf} = \infty$, then ${\pep[\sgp](\gp) \dfeq \dst(\gp, \gmnf)}$ (this also covers the case $\md = \ad$ since then necessarily $\gmnf = \RR^\ad$). Otherwise, if $\gmnf$ has dimension~$0$, then ${\pep[\sgp](\gp) \dfeq \dst(\gp, \sgp)}$. If both the dimension and codimension of~$\gmnf$ are positive and $\rch{\gmnf} < \infty$, we split the definition of $\pep[\sgp]$ into two cases. Let $\pep[\sgp](\sgp) \dfeq 0$. For $\gp \in \RR^\ad \setminus \set{\sgp}$ introduce a tangent-normal coordinate system with the origin in~$\sgp$ (it exists by Lemma~\ref{lemma:general-manifold-properties}(\ref{lemma:general-manifold-properties:planar-tangent-normal-coordinate-system})). Let $\gp = (x_1, \ldots, x_\ad)$ be the coordinates of~$\gp$ in this coordinate system. Define $\pep[\sgp](\gp)$ to be the unique element in~$\RR_{> 0}$ which satisfies the equation
			\[\frac{x_1^2 + \ldots + x_\md^2}{\rch{\gmnf} \, \pep[\sgp](\gp) + \pep[\sgp](\gp)^2} + \frac{x_{\md+1}^2 + \ldots + x_\ad^2}{\pep[\sgp](\gp)^2} = 1.\]
		\end{definition}
		Since the sum of squares of coordinates is independent of the choice of an orthonormal coordinate system, this equation depends only on $\gp$ and~$\sgp$. Lemma~\ref{lemma:inside-ellipsoid} guarantees existence, uniqueness and continuity of~$\pep[\sgp](\gp)$.
		
		The point of this definition is that (except in the case $\md = \ad$, when all ellipsoids are the whole~$\RR^\ad$) the unique ellipsoid of the form~$\cf{\sgp}{r}$ which has $\gp$ in its boundary has $r = \pep[\sgp](\gp)$, i.e.~$\gp \in \bd{\cf{\sgp}{\pep[\sgp](\gp)}}$.
		
		\begin{lemma}\label{lemma:distance-to-manifold}
			Let $\gmnf$ be a properly embedded $\C{1}$-submanifold of~$\RR^\ad$. Let $\gp \in \RR^\ad$ and $\sgp \in \gmnf$. Then $\dst(\gmnf, \gp) \leq \pep[\sgp](\gp)$.
		\end{lemma}
		
		\begin{proof}
			If $\rch{\gmnf} = \infty$, the statement is clear, so assume $\rch{\gmnf} < \infty$.
			
			Let $\md$ be the dimension of~$\gmnf$ at~$\sgp$. If $\md = 0$, then $\dst(\gmnf, \gp) \leq \dst(\sgp, \gp) = \pep[\sgp](\gp)$.
			
			For $0 < \md < \ad$ we rely on Lemma~\ref{lemma:general-manifold-properties}. There is a planar tangent-normal coordinate system which has the origin in~$\sgp$ and contains~$\gp$. We can additionally assume that the axes are oriented so that $\gp$ is in the closed first quadrant. Since $\gp \in \bd{\cf{\sgp}{\pep[\sgp](\gp)}}$, there exists $\varphi \in \intcc{0}{\frac{\pi}{2}}$ such that the coordinates of~$\gp$ in this coordinate system are
			\[\big(\sqrt{\rch{\gmnf} q + q^2} \, \cos(\varphi), q \, \sin(\varphi)\big),\]
			where we have shortened $q \dfeq \pep[\sgp](\gp)$. Hence
			\[\dst(\gmnf, \gp) \leq \sqrt{(\rch{\gmnf} q + q^2) \cos^2(\varphi) + \big(q \, \sin(\varphi) + \rch{\gmnf}\big)^2} - \rch{\gmnf} =\]
			\[= \sqrt{\rch{\gmnf}^2 + q^2 + \rch{\gmnf} q \big(\cos^2(\varphi) + 2 \sin(\varphi)\big)} - \rch{\gmnf} =\]
			\[= \sqrt{\rch{\gmnf}^2 + q^2 + \rch{\gmnf} q \big(1 - \sin^2(\varphi) + 2 \sin(\varphi)\big)} - \rch{\gmnf} =\]
			\[= \sqrt{\rch{\gmnf}^2 + q^2 + \rch{\gmnf} q \big(2 - (1 - \sin(\varphi))^2\big)} - \rch{\gmnf}.\]
			Clearly, the last expression is the largest where the function $\intcc{0}{\frac{\pi}{2}} \to \RR$, $\varphi \mapsto 2 - (1 - \sin(\varphi))^2$ attains a maximum which is at $\varphi = \tfrac{\pi}{2}$. Thus the distance $\dst(\gmnf, \gp)$ is the largest in the normal space at~$\sgp$, where we get
			\[\dst(\gmnf, \sgp) \leq \sqrt{\rch{\gmnf}^2 + q^2 + \rch{\gmnf} q \big(2 - (1 - 1)^2\big)} - \rch{\gmnf} = \sqrt{\rch{\gmnf}^2 + q^2 + 2 \rch{\gmnf} q} - \rch{\gmnf} = q.\]
		\end{proof}
		
		Let us also recall some facts about Lipschitz maps that we will need later. A map~$f$ between subsets of Euclidean spaces is \df{Lipschitz} when it has a \df{Lipschitz coefficient} $C \in \RR_{\geq 0}$, so that for all $\gp, \sgp$ in the domain of~$f$ we have $\big\|f(\gp) - f(\sgp)\big\| \leq C \cdot \|\gp - \sgp\|$. A function is \df{locally Lipschitz} when every point of its domain has a neighbourhood such that the restriction of the function to this neighbourhood is Lipschitz.
		
		Let $f$ and~$g$ be maps with Lipschitz coefficients $C$ and~$D$, respectively. Then clearly $C+D$ is a Lipschitz coefficient for the functions $f+g$ and~$f-g$, and $C \cdot D$ is a Lipschitz coefficient for~$g \circ f$ (whenever these functions exist).
		
		For bounded functions the Lipschitz property is preserved under further operations. A function being bounded is meant in the usual way, i.e.~being bounded in norm.
		
		\begin{lemma}\label{lemma:lipschitz}
			Let $f$ and $g$ be maps between subsets of Euclidean spaces with the same domain. Assume that $f$ and $g$ are bounded and Lipschitz.
			\begin{enumerate}
				\item
					If $b$ is bilinear with the property $\big\|b(\gp, \sgp)\big\| \leq \|\gp\| \, \|\sgp\|$, then the map $\gp \mapsto b\big(f(\gp), g(\gp)\big)$ is bounded Lipschitz.\footnote{In practice, $b$ is the product of numbers or scalar product of vectors.}
				\item
					Assume $g$ takes values in~$\RR$ and has a positive lower bound $m \in \RR_{> 0}$. Then the map $x \mapsto \frac{f(\gp)}{g(\gp)}$ is bounded Lipschitz.
			\end{enumerate}
		\end{lemma}
		
		\begin{proof}
			Let $M$ be an upper bound for the norms of~$f$ and~$g$ and let $C$ be a Lipschitz coefficient for~$f$ and~$g$. Let $\gp$, $\gp'$ and $\gp''$ be elements of the domain of $f$ and~$g$.
			\begin{enumerate}
				\item
					Boundedness: $\displaystyle{\big\|b\big(f(\gp), g(\gp)\big)\big\| \leq \big\|f(\gp)\big\| \, \big\|g(\gp)\big\| \leq M^2}$.
					
					Lipschitz property:
					\begin{align*}
						&\big\|b\big(f(\gp'), g(\gp')\big) - b\big(f(\gp''), g(\gp'')\big)\big\| \\
						&\quad= \big\|b\big(f(\gp'), g(\gp')\big) - b\big(f(\gp''), g(\gp')\big) + b\big(f(\gp''), g(\gp')\big) - b\big(f(\gp''), g(\gp'')\big)\big\| \\
						&\quad\leq \big\|f(\gp') - f(\gp'')\big\| \, \big\|g(\gp')\big\| + \big\|f(\gp'')\big\| \, \big\|g(\gp') - g(\gp'')\big\| \\
						&\quad\leq 2 C M \big\|\gp' - \gp''\big\|.
					\end{align*}
				\item
					Boundedness: $\displaystyle{\Big\|\frac{f(\gp)}{g(\gp)}\Big\| = \frac{\|f(\gp)\|}{|g(\gp)|} \leq \frac{M}{m}}$.
					
					Lipschitz property:
					\begin{align*}
						\Big\|\frac{f(\gp')}{g(\gp')} - \frac{f(\gp'')}{g(\gp'')}\Big\| &= \Big\|\frac{f(\gp') g(\gp'') - f(\gp'') g(\gp')}{g(\gp') g(\gp'')}\Big\| \\
						&= \frac{\|f(\gp') g(\gp'') - f(\gp'') g(\gp'') + f(\gp'') g(\gp'') - f(\gp'') g(\gp')\|}{|g(\gp') g(\gp'')|} \\
						&\leq \frac{\|f(\gp') - f(\gp'')\| \, \|g(\gp'')\| + \|f(\gp'')\| \, \|g(\gp'') - g(\gp')\|}{|g(\gp')| \, |g(\gp'')|} \\
						&\leq \frac{2 C M}{m^2} \big\|\gp' - \gp''\big\|.
					\end{align*}
			\end{enumerate}
		\end{proof}
		
		\begin{corollary}\label{corollary:partition-of-unity-locally-lipschitz-amalgamation-is-locally-lipschitz}
			Let $(U_i)_{i \in I}$ be a locally finite open cover of a subset~$U$ of a Euclidean space, $(f_i)_{i \in I}$ a subordinate smooth partition of unity and $(g_i\colon U_i \to \RR^\ad)_{i \in I}$ a family of maps. Let $g\colon U \to \RR^\ad$ be the map, obtained by gluing maps~$g_i$ with the partition of unity~$f_i$, i.e.
			\[g(x) \dfeq \sum_{i \in I} f_i(x) \;\! g_i(x).\]
			Then if all~$g_i$ are locally Lipschitz, so is~$g$.
		\end{corollary}
		
		\begin{proof}
			Every continuous map is locally bounded, including the derivative of a smooth map, the bound on which is then a local Lipschitz coefficient for the map. We can apply this for~$f_i$.
			
			Given $x \in U$, pick an open set $V \subseteq U$, for which the following holds: $x \in V$, there is a finite set of indices $F \subseteq I$ such that $V$ intersects only~$U_i$ with $i \in F$ and $V \subseteq \bigcap_{i \in F} U_i$, and the maps $f_i$ and~$g_i$ are bounded and Lipschitz on~$V$ for every $i \in F$. Then $\rstr{g}_V = \sum_{i \in F} \rstr{f_i}_V \;\! \rstr{g_i}_V$ which is Lipschitz on~$V$ by Lemma~\ref{lemma:lipschitz}.
		\end{proof}

	\section{Calculating Bounds on Persistence Parameter}\label{section:bounds-on-persistence-parameter}
		
		Having derived some results for more general manifolds, we now specify the manifolds for which our main theorem holds. We reserve the symbol~$\mnf$ for such a manifold.
		
		Let $\mnf$ be a non-empty $\md$-dimensional properly embedded $\C{1}$-submanifold of~$\RR^\ad$ without boundary, and let~$\ma$ be its medial axis. Let $\rch$ denote the reach of~$\mnf$. In this section we assume $\rch < \infty$ and in Sections~\ref{section:program} and~\ref{section:deformation-retraction-construction} we assume $\rch = 1$. We will drop these assumptions on~$\rch$ for the main theorem in Section~\ref{section:main-result}.
		
		By Proposition~\ref{proposition:properly-embedded-submanifolds} and the definition of a medial axis the map $\pr\colon \RR^\ad \setminus \ma \to \mnf$, which takes a point to its closest point on the manifold~$\mnf$, is well defined. We also define $\prv\colon \RR^\ad \setminus \ma \to \mnf$, $\prv(\gp) \dfeq \pr(\gp) - \gp$. We view~$\prv(\gp)$ as the vector, starting at~$\gp$ and ending in~$\pr(\gp)$. This vector is necessarily normal to the manifold, i.e.~it lies in~$\ns{\pr(\gp)}$. By the definition of the reach, the maps $\pr$ and~$\prv$ are defined on~$\otnm{\rch}$.
		
		\begin{lemma}\label{lemma:continuity-of-closest-point-projection}
			For every $r \in \intoo{0}{\rch}$ the maps $\pr$ and~$\prv$ are Lipschitz when restricted to $\ctnm{r}$, with Lipschitz coefficients $\frac{\rch}{\rch - r}$ and $\frac{\rch}{\rch - r} + 1$, respectively. Hence these two maps are continuous on~$\otnm{\rch}$.
		\end{lemma}
		
		\begin{proof}
			The map $\pr$ is Lipschitz on~$\ctnm{r}$ by~\cite[Proposition~2]{Chazal2017} with a Lipschitz coefficient~$\frac{\rch}{\rch - r}$~\cite[Theorem~4.8(8)]{federer1959curvature}. As a difference of two Lipschitz maps, the map~$\prv$ is Lipschitz as well, with a Lipschitz coefficient~$\frac{\rch}{\rch - r} + 1$. The maps $\pr$ and~$\prv$ are therefore continuous on $\otnm{r}$ for all $r \in \intoo{0}{\rch}$, and hence also on the union ${\otnm{\rch} = \bigcup_{r \in \intoo{0}{\rch}} \otnm{r}}$.
		\end{proof}
		
		We want to approximate the manifold~$\mnf$ with a sample. We assume that the sample set~$\smp$ is a non-empty discrete subset of~$\mnf$, locally finite in~$\RR^\ad$ (meaning, every point in~$\RR^\ad$ has a neighbourhood which intersects only finitely many points of~$\smp$). It follows that $\smp$ is a closed subset of~$\RR^\ad$.
		
		Let $\hd$ denote the Hausdorff distance between $\mnf$ and~$\smp$. We assume that $\hd$ is finite. This value represents the density of our sample: it means that every point on the manifold~$\mnf$ has a point in the sample~$\smp$ which is at most $\hd$ away.
		
		Since $\mnf$ is properly embedded in~$\RR^\ad$ and $\hd < \infty$, the sample~$\smp$ is finite if and only if $\mnf$ is compact. A properly embedded non-compact submanifold without boundary needs to extend to infinity and so cannot be sampled with finitely many points (think for example about the hyperbola in the plane, $x^2 - y^2 = 1$). As it turns out, we do not need finiteness, only local finiteness, to prove our results.
		
		If the sample is dense enough in the manifold, it should be a good approximation to it. Specifically, we want to recover at least the homotopy type of~$\mnf$ from the information, gathered from~$\smp$. A common way to do this is to enlarge the sample points to balls, the union of which deformation retracts to the manifold, so has the same homotopy type (in other words, we consider a \v{C}ech complex of the sample).
		
		In this paper we use ellipsoids instead of balls. The idea is that a tangent space at some point is a good approximation for the manifold at that point, so an ellipsoid with the major semi-axes in the tangent directions should better approximate the manifold than a ball. Consequently we should require a less dense sample for the approximation. This idea indeed pans out (as demonstrated by Theorem~\ref{section:main-result}), though it turns out that the standard methods, used to construct the deformation retraction from the union of balls to the manifold, do not work for the ellipsoids.
		
		Given a persistence parameter $\pp \in \RR_{> 0}$, let us denote the unions of open and closed tangent-normal $\pp$-ellipsoids around sample points by
		\begin{align*}
			\on &\dfeq \bigcup_{\smpp \in \smp} \of, \\
			\cn &\dfeq \bigcup_{\smpp \in \smp} \cf.
		\end{align*}
		As a union of open sets, $\on$ is open in~$\RR^\ad$. As a locally finite union of closed sets, $\cn$ is closed in~$\RR^\ad$.
		
		We want a deformation retraction from $\on$ to~$\mnf$. Clearly this will not work for all $\pp \in \RR_{> 0}$. If $\pp$ is too small, $\on$ covers only some blobs around sample points, not the whole~$\mnf$. If $\pp$ is too large, $\on$ reaches over the medial axis~$\ma$, therefore creates connections which do not exist in the manifold, so differs from it in the homotopy type. This suggests that the lower bound on~$\pp$ will be expressed in terms of~$\hd$ (the denser the sample, the smaller the required~$\pp$ for $\on$ to cover~$\mnf$), and the upper bound on~$\pp$ will be expressed in terms of~$\rch$ (the further away the medial axis, the larger we can make the ellipsoids so that they still do not intersect the medial axis).
		
		\begin{lemma}\label{lemma:persistence-parameter-lower-bound}
			\
			\begin{enumerate}
				\item
					Assume $\pp \in \RR_{> 0}$ satisfies $\hd < \sqrt{2\pp \big(\sqrt{\rch (\pp + 2\rch)} - \rch\big)}$. Then $\mnf \subseteq \on$, i.e.~$\big(\of\big)_{\smpp \in \smp}$ is an open cover of~$\mnf$.
				\item
					The map $\RR_{> 0} \to \RR_{> 0}$, $\pp \mapsto \sqrt{2\pp \big(\sqrt{\rch (\pp + 2\rch)} - \rch\big)}$, is strictly increasing. Thus there exists a unique $\lppb \in \RR_{> 0}$ such that
					\[\hd < \sqrt{2\pp \big(\sqrt{\rch (\pp + 2\rch)} - \rch\big)} \iff \lppb < \pp.\]
			\end{enumerate}
		\end{lemma}
		
		\begin{proof}
			\
			\begin{enumerate}
				\item
					Take any $\gp \in \mnf$. By assumption there exists $\smpp \in \smp$ such that $\dst(\gp, \smpp) \leq \hd$. We claim that $\gp \in \of$.
					
					If $\md = \ad$, then  $\of = \ob{\smpp}{\sqrt{\rch \pp + \pp^2}}$, and a quick calculation shows that
					\[\sqrt{2\pp \big(\sqrt{\rch (\pp + 2\rch)} - \rch\big)} \leq \sqrt{\rch \pp + \pp^2},\]
					so $\gp \in \of$.
					
					If $\md = 0$, then $\of = \ob{\smpp}{\pp}$ and the reach~$\rch$ is half of the distance between the two closest distinct points in~$\mnf$ (since we are assuming $\rch < \infty$ and therefore $\ma \neq \emptyset$, the manifold~$\mnf$ must have at least two points). If $\pp \leq 2\rch$, then
					\[\sqrt{2\pp \big(\sqrt{\rch (\pp + 2\rch)} - \rch\big)} \leq 2\rch,\]
					so necessarily $\gp = \smpp \in \of$. If $\pp > 2\rch$, then
					\[\sqrt{2\pp \big(\sqrt{\rch (\pp + 2\rch)} - \rch\big)} < \pp,\]
					so $\gp \in \ob{\smpp}{\pp} = \of$.
					
					Assume hereafter that $0 < \md < \ad$. Choose a planar tangent-normal coordinate system with the origin in~$\smpp$ which contains~$\gp$ (use Lemma~\ref{lemma:general-manifold-properties}(\ref{lemma:general-manifold-properties:planar-tangent-normal-coordinate-system})). In this coordinate system the boundary of~$\cf$ is given by the equation $\frac{x^2}{\rch \pp + \pp^2} + \frac{y^2}{\pp^2} = 1$. A routine calculation shows that it intersects the boundaries of the $\tau$-balls, associated to~$\smpp$ (with centers in $\cab' = (0, \rch)$ and $\cab'' = (0, -\rch)$), given by the equations $x^2 + (y \pm \rch)^2 = \rch^2$, in the points
					\[\Big(\pm\sqrt{\frac{\pp (\pp + \rch) \big(2 \sqrt{\rch (\pp + 2\rch)} - \pp - \rch\big)}{\rch}}, \pm\frac{\pp \big(\sqrt{\rch (\pp + 2\rch)} - \rch\big)}{\rch}\Big),\]
					the norm of which is~$r \dfeq \sqrt{2\pp \big(\sqrt{\rch (\pp + 2\rch)} - \rch\big)} > \hd \geq \dst(\gp, \smpp)$. It follows that within the given two-dimensional coordinate system
					\[\gp \in \ob{\smpp}{r} \subseteq \of \cup \ob{\cab'}{\rch} \cup \ob{\cab''}{\rch},\]
					see Figure~\ref{figure:covering-manifold-with-ellipsoids}.
					\begin{figure}[!ht]
						\centering
						\includegraphics[width=0.7\textwidth]{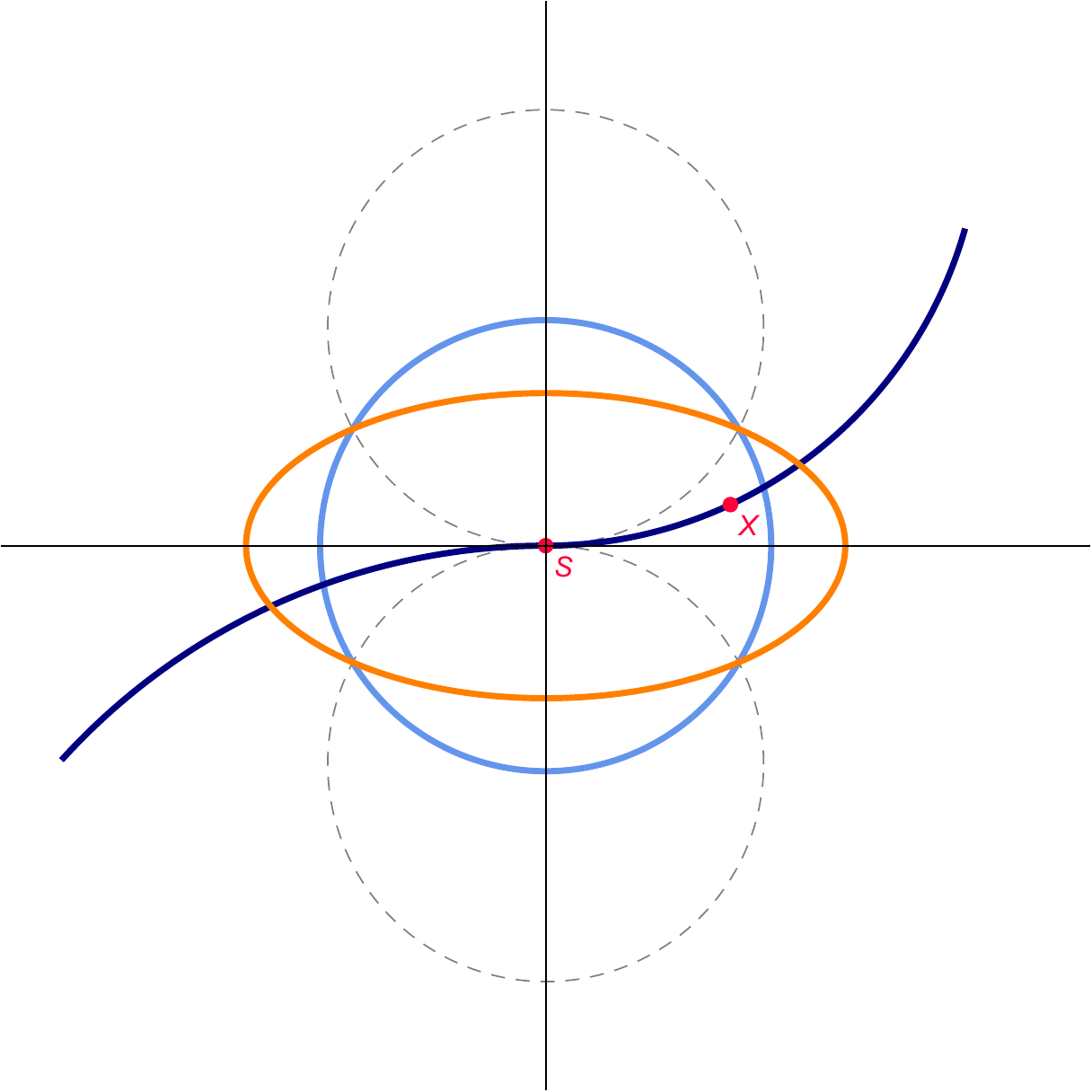}
						\caption{The point~$X$ within the ellipsoid.}\label{figure:covering-manifold-with-ellipsoids}
					\end{figure}
					
					Since $\smpp \in \mnf$ and the reach of~$\mnf$ is~$\rch$, the manifold~$\mnf$ does not intersect the open $\tau$-balls, associated to~$\smpp$, so $\gp \in \of$.
				\item
					The derivative of the given function is
					\[\frac{\rch \big(3\pp + 4\rch - 2\sqrt{\rch (\pp + 2\tau)}\big)}{2 \sqrt{2 \pp \rch (\pp + 2\rch) \big(\sqrt{\rch (\pp + 2\rch)} - \rch\big)}},\]
					which is positive for $\pp, \rch > 0$ which assures the existence of the required~$\lppb$. Calculated with \textit{Mathematica}, the actual value is
					\[\lppb = \frac{2 \tau  \left(3 \kappa ^2+\tau ^2\right)}{3 \sqrt[3]{27 \kappa ^4 \tau ^2-36 \kappa
					   ^2 \tau ^4+3 \sqrt{81 \kappa ^8 \tau ^4-408 \kappa ^6 \tau ^6-96 \kappa ^4 \tau ^8}-8
					   \tau ^6}}+\]
					\[\frac{\sqrt[3]{27 \kappa ^4 \tau ^2-36 \kappa ^2 \tau ^4+3 \sqrt{81 \kappa
					   ^8 \tau ^4-408 \kappa ^6 \tau ^6-96 \kappa ^4 \tau ^8}-8 \tau ^6}}{6 \tau }-\frac{\tau
					   }{3}.\]
			\end{enumerate}
		\end{proof}

		We can strengthen this result to thickenings of~$\mnf$. Given $r \in \RR_{> 0}$, we denote the open and closed $r$-thickening of~$\mnf$ by
		\begin{align*}
			\otnm{r} &\dfeq \set{\gp \in \RR^\ad}{\dst(\mnf, \gp) < r}, \\
			\ctnm{r} &\dfeq \set{\gp \in \RR^\ad}{\dst(\mnf, \gp) \leq r}.
		\end{align*}
		
		\begin{corollary}
			For every $r \in \RR_{\geq 0}$ and every $\pp \in \RR_{> \lppb + r}$ we have $\ctnm{r} \subseteq \on$.
		\end{corollary}
		
		\begin{proof}
			Lemma~\ref{lemma:persistence-parameter-lower-bound} implies that $\mnf \subseteq \on{\pp-r}$. Hence $\ctnm{r}$ is contained in the union of $r$-thickenings of open ellipsoids~$\of{\smpp}{\pp-r}$, and an $r$-thickening of~$\of{\smpp}{\pp-r}$ is contained in~$\of$.
		\end{proof}
		
		Let us now also get an upper bound on~$\pp$.
		
		\begin{lemma}
			Assume $\pp \in \intoo{0}{\rch}$. Then $\cn \subseteq \otnm{\rch}$; in particular $\on$ and $\cn$ do not intersect the medial axis of~$\mnf$.
		\end{lemma}
		
		\begin{proof}
			Take any $\smpp \in \smp$ and $\gp \in \cf$. By Lemma~\ref{lemma:distance-to-manifold} we have $\dst(\mnf, \gp) \leq \pep(\gp) \leq \pp < \rch$.
		\end{proof}
		
		The results in this section give the theoretical bounds on the persistence parameter~$\pp$, within which we look for a deformation retraction from~$\on$ to~$\mnf$, which we summarize in the following corollary.
		
		\begin{corollary}\label{corollary:persistence-parameter-bounds}
			If $\pp \in \intoo{\lppb}{\rch}$, then $\mnf \subseteq \on \subseteq \ma^\complement$.
		\end{corollary}

	\section{Program}\label{section:program}
		
		In this section (as well as the next one) we assume that $\rch = 1$ and $0 < \md < \ad$.
		
		Our goal is to prove that if we restrict the persistence parameter~$\pp$ to a suitable interval, the union of ellipsoids~$\on$ deformation retracts to~$\mnf$. Recall that the normal deformation retraction is the map retracting a point to its closest point on the manifold, i.e.~the convex combination of a point and its projection: ${(\gp, t) \mapsto (1-t) \;\! \gp + t \;\! \pr(\gp) = \gp + t \;\! \prv(\gp)}$. For example, in~\cite{NSW} this is how the union of balls around sample points is deformation retracted to the manifold.
		
		The same idea does not in general work for the union of ellipsoids, or any other sufficiently elongated figures. Figure~\ref{figure:normal-projection-not-working} shows what can go wrong.
		\begin{figure}[!ht]
			\centering
			\includegraphics[width=0.9\textwidth]{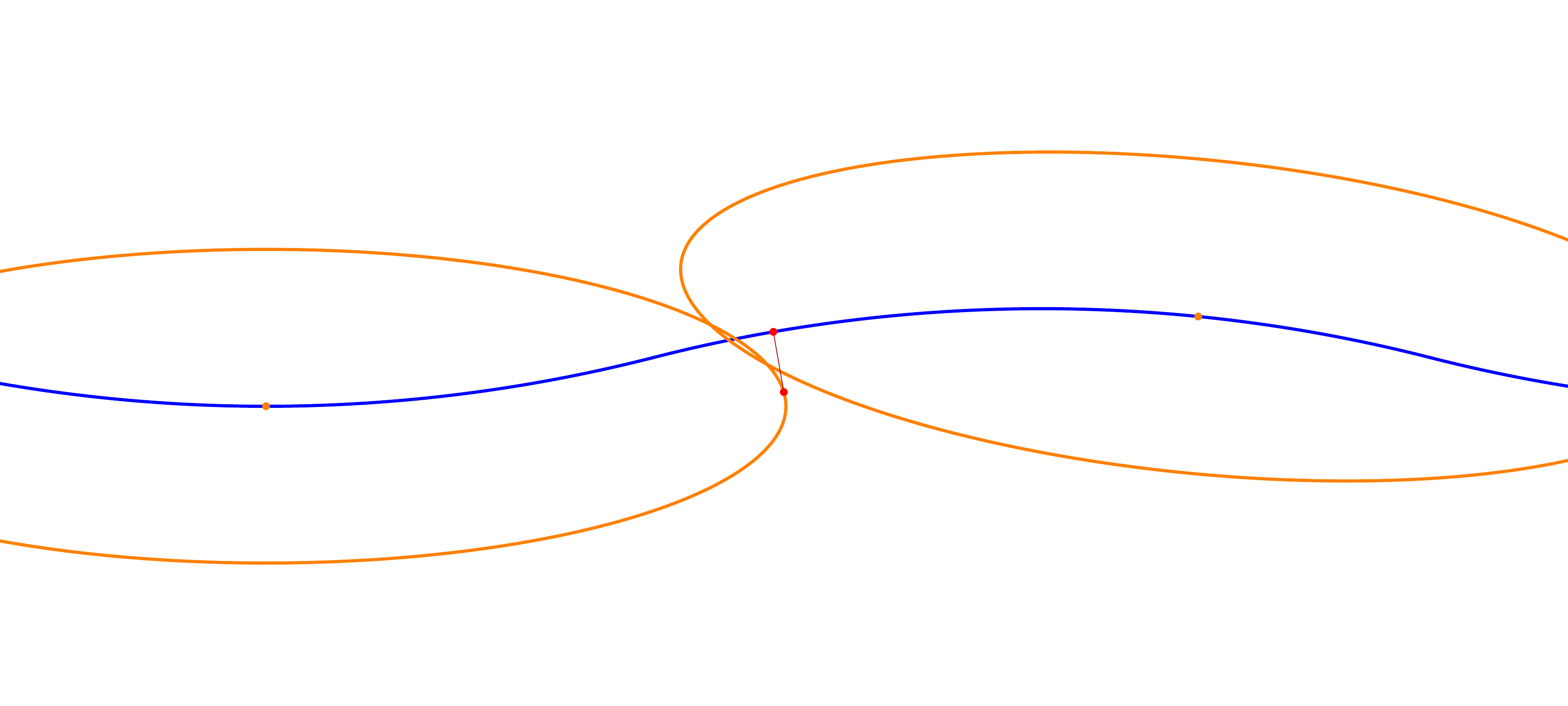}
			\caption{Normal deformation retraction does not always work.}\label{figure:normal-projection-not-working}
		\end{figure}
		
		However, it turns out that the only places where the normal deformation retraction does not work are the neighbourhoods of tips of some ellipsoids which avoid all other ellipsoids. This section is dedicated to proving the following form of this claim: for all points in at least two ellipsoids the normal deformation retraction works. This means that the line segment between a point~$\gp$ and $\pr(\gp)$ is contained in the union of ellipsoids, but actually more holds: the line segment is contained already in one of the ellipsoids. More formally, the rest of the section is the proof of the following lemma.
		
		\begin{lemma}\label{lemma:program-conclusion}
			For every $\gp \in \on$, if there are $\smpp', \smpp'' \in \smp$, $\smpp' \neq \smpp''$ such that ${\gp \in \of{\smpp'} \cap \of{\smpp''}}$, then there exists $\smpp \in \smp$ such that $\gp, \pr(\gp) \in \of$. By convexity the entire line segment between $\gp$ and~$\pr(\gp)$ is therefore in~$\of$.
		\end{lemma}
		
		To prove this, we would in principle need to examine all possible configurations of ellipsoids and a point. However, we can restrict ourselves to a set of cases, which include the ``worst case scenarios''.
		
		Let $\smpp', \smpp'' \in \smp$ be two different sample points, and let $\gp \in \cf{\smpp'} \cap \cf{\smpp''}$ (we purposefully take closed ellipsoids here). Denote $\sgp \dfeq \pr(\gp)$. We claim that there is $\smpp \in \smp$ (not necessarily distinct from $\smpp'$ and~$\smpp''$) such that $\gp \in \cf$ and $\sgp \in \of$. Due to convexity of ellipsoids, the line segment $\gp\sgp$ is in~$\cf$; with the possible exception of the point~$\gp$, this line segment is in~$\of$.
		
		Assuming $\pp \in \intoo{\lppb}{1}$, the point $\sgp$ is covered by at least one open ellipsoid. Suppose that none of the closed ellipsoids, containing~$\sgp$ in their interior, contains~$\gp$. Let us try to construct a situation where this is most likely to be the case. We will derive a contradiction by showing that even in these ``worst case scenarios'' we fail in satisfying this assumption.
		
		To determine whether a point~$\gp$ is in the ellipsoid with the center~$\smpp'$, the following two pieces of information are sufficient: the distance between~$\gp$ and~$\smpp'$, and the angle between the line segment $\gp\smpp'$ and the normal space~$\ns{\smpp'}$. Moreover, membership of~$\gp$ in the ellipsoid is ``monotone'' with respect to these two conditions: if a point is in the ellipsoid, it will remain so if we decrease its distance to~$\smpp'$ or increase the angle to the normal space.
		
		We will produce a set of configurations which include the extremal points for these two criteria (maximal distance from the ellipsoid center, minimal angle to the normal space). If every such point is still in the ellipsoid, then all possible points are.
		
		Consider a planar tangent-normal coordinate system with the origin in~$\smpp'$ which contains~$\gp$ in the fourth quadrant (nonnegative tangent coordinate, nonpositive normal coordinate). In this coordinate system, the manifold passes horizontally through~$\smpp'$. Consider the part of the manifold with positive tangent coordinate (i.e.~the part of the manifold rightwards of~$\smpp'$). The fastest that this piece can turn away from~$\gp$ is in this plane along the boundary of the upper $\rch$-ball, associated to~$\smpp'$.\footnote{Imagine distinct points $\pt{A}$ and~$\pt{B}$ in some higher-dimensional Euclidean space and a non-zero vector~$\vec{a}$, starting at~$\pt{A}$ and having a nonnegative scalar product with~$\vec{AB}$. Consider paths, starting at~$\pt{A}$ and going in the direction of~$\vec{a}$, the curvature of which is bounded by some number. Then the fastest that we can get away from~$B$ along such a path is within the plane, determined by~$\pt{A}$, $\pt{B}$ and~$\vec{a}$ --- specifically, along the arc with the maximum curvature.} Suppose the manifold continues along this path until some point~$\gp'$, and consider a plane containing the points $\gp$, $\gp'$ and~$\smpp$ where the distance between $\smpp \in \smp$ and~$\sgp$ is bounded by~$\hd$, so $\sgp \in \of$. Going from~$\gp'$ to~$\smpp$, the quickest way to turn the normal direction towards~$\gp$ is within this plane, and along a $\rch$-arc. While this second plane need not be the same as the first one, they intersect along the line containing~$\gp$ and~$\gp'$. We can turn the half-plane containing~$\smpp'$ and the half-plane containing~$\smpp$ along the line so that they form one plane, and that will be the configuration where it is equally (un)likely for $\of$ to contain~$\gp$, but where $\smpp'$, $\gp$, $\gp'$, $\sgp$ and~$\smpp$ all lie in the same plane.
		
		We can make the same argument starting from~$\smpp''$ instead of~$\smpp'$, so we conclude the following: if our claim fails for some configuration of~$\gp$, $\sgp$, $\smpp'$, $\smpp''$, $\smpp$, then it fails in a planar case where the part of the manifold connecting points~$\smpp'$ and~$\smpp''$ consists of (at most) three $\rch$-arcs, as in Figure~\ref{figure:three-ellipsoids-configuration}.
		\begin{figure}[!htp]
			\centering
			\includegraphics[width=0.8\textwidth]{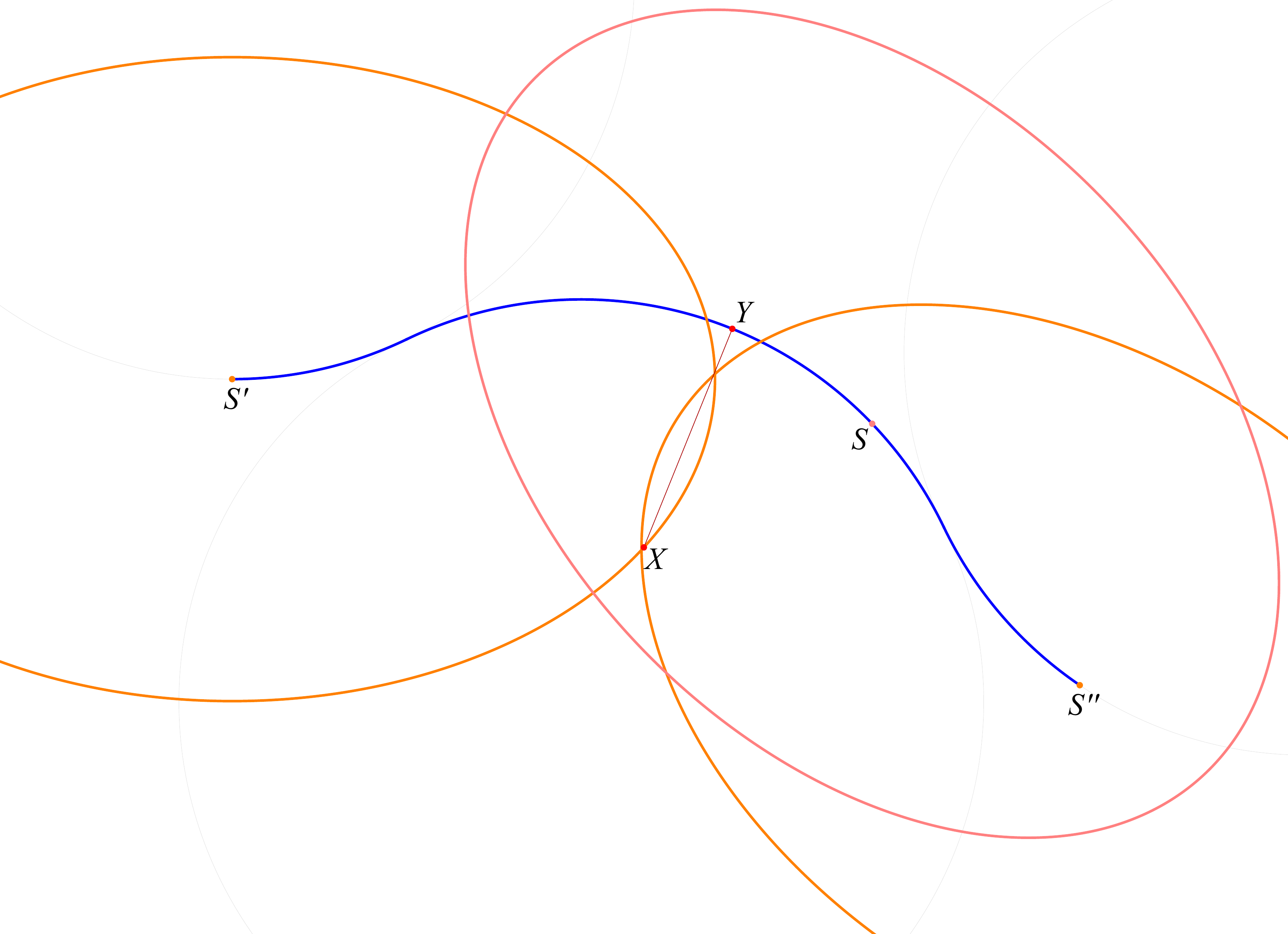}
			\caption{Point in two ellipsoids, whose projection is in another ellipsoid}\label{figure:three-ellipsoids-configuration}
		\end{figure}
		
		We started with the assumption $\gp \in \cf{\smpp'}{\pp} \cap \cf{\smpp''}{\pp}$, but we may without loss of generality additionally assume $\gp \in \bd\cf{\smpp'}{\pp}$. If we had a counterexample~$\gp$ to our claim in the interior of all ellipsoids containing~$\gp$, we could project it in the opposite direction of~$\pr(\gp)$ to the first ellipsoid boundary we hit, and declare the center of that ellipsoid to be~$\smpp'$.
		
		Although the reduction of cases we have made is already a vast simplification of the necessary calculations, we find that it is still not enough to make a theoretical derivation of the desired result feasible. Instead, we produce a proof with a computer.
		
		We can reduce the possible configurations to four parameters (see Figure~\ref{figure:program-ellipsoids-configuration}):
		\begin{itemize}
			\item
				$\alpha$ denotes the angle measuring the length of the first $\rch$-arc,
			\item
				$\sigma$ denoted the angle for the second $\rch$-arc until~$\smpp$,
			\item
				$\pp$ is, as usual, the persistence parameter,
			\item
				$\chi$ determines the position of~$\gp$ in the boundary~$\bd\cf{\smpp'}$.
		\end{itemize}
		\begin{figure}[!htp]
			\centering
			\includegraphics[width=0.8\textwidth]{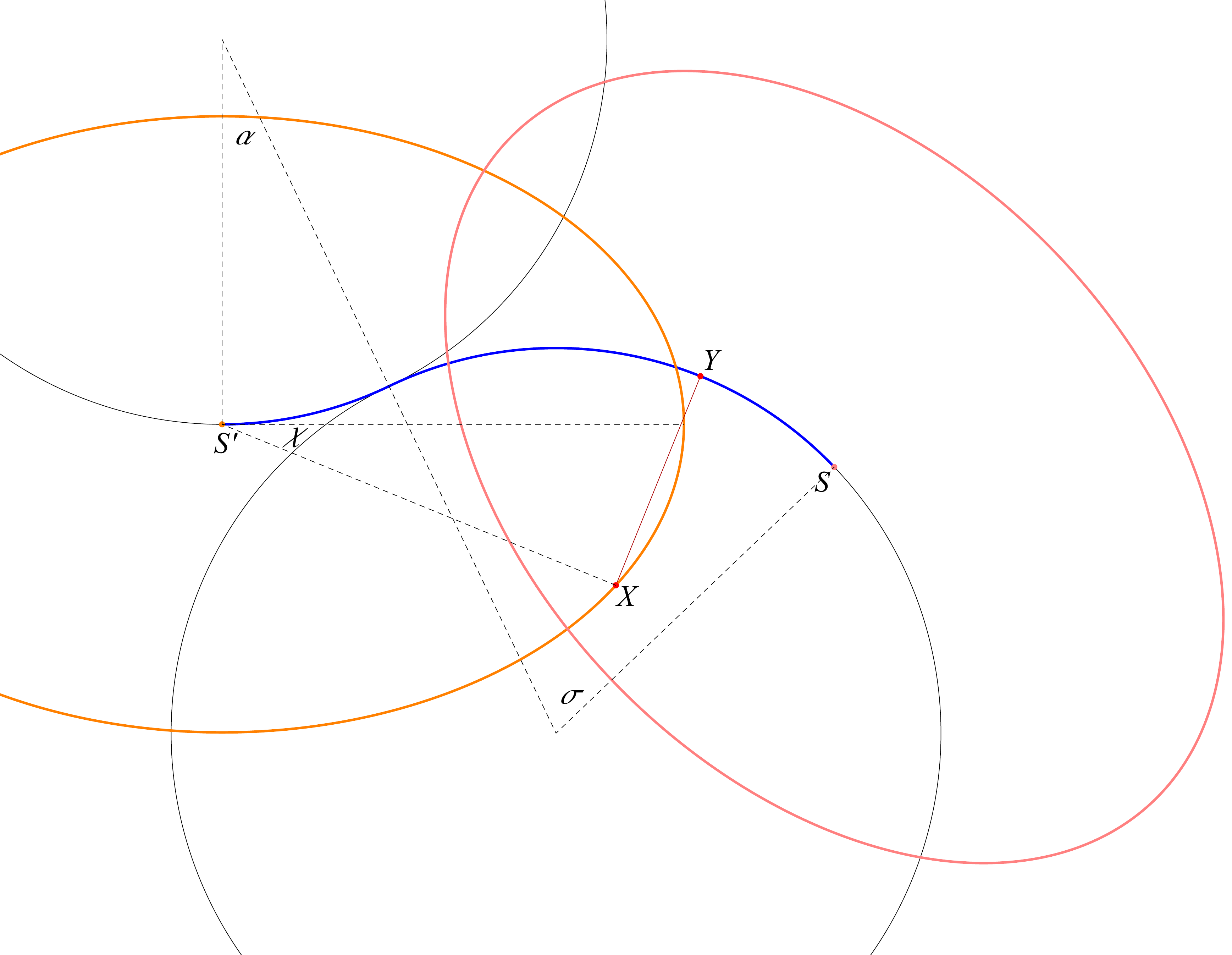}
			\caption{Notation of parameters in the program}\label{figure:program-ellipsoids-configuration}
		\end{figure}
		Notice that Figure~\ref{figure:program-ellipsoids-configuration} does not include both ellipsoids containing~$\gp$ but not~$\sgp$, like Figure~\ref{figure:three-ellipsoids-configuration} does. It turns out that as soon as~$\sgp$ is not in the first ellipsoid, both $\gp$ and~$\sgp$ will be in an ellipsoid, the center of which is within $\hd$~distance from~$\sgp$. This allows us to restrict ourselves to just the four aforementioned variables, which makes the program run in a reasonable time.
		
		The space of the configurations we restricted ourselves to --- let us denote it by~$\cs$ --- is compact (we give its precise definition below). We want to calculate for each configuration in~$\cs$ that $\gp$ is in some ellipsoid with the center within $\hd$~distance from~$\sgp$ (it follows automatically that $\sgp$ is in this ellipsoid). The boundary of the ellipsoid is a level set of a smooth function. We can compose it with a suitable linear function so that $\gp$ is in the open ellipsoid if and only if the value of the adjusted function is positive. Let us denote this adjusted function by $\pf\colon \cs \to \RR$; we have our claim if we show that~$\pf$ is positive for all configurations in~$\cs$.
		
		Of course, the program cannot calculate the function values for all infinitely many configurations in~$\cs$. We note that the (continuous) partial derivatives of~$\pf$ are bounded on compact~$\cs$, hence the function is Lipschitz. If we change the parameters by at most~$\delta$, the function value changes by at most~$C \cdot \delta$ where $C$ is the Lipschitz coefficient. The program calculates the function values in a finite lattice of points, so that each point in~$\cs$ is at most a suitable~$\delta$ away from the lattice, and verifies that all these values are larger than~$C \cdot \delta$. This shows that~$\pf$ is positive on the whole~$\cs$.
		
		Let us now define~$\cs$ precisely and then calculate the Lipschitz coefficient of~$\pf$. We may orient the coordinate system so that the point~$\gp$ is in the closed fourth quadrant. Hence we have $\gp = \big(\sqrt{\pp + \pp^2} \cos(\chi), -\pp \, \sin(\chi)\big)$, where $\chi$ ranges over the interval~$\intcc{0}{\frac{\pi}{2}}$.
		
		Unfortunately due to our method we cannot allow~$\pp$ to range over the whole interval~$\intoo{\lppb}{1}$; if we did, the values of~$\pf$ would come arbitrarily close to zero, in particular below~$C \cdot \delta$, so the program would not prove anything. Let us set $\pp \in \intcc{m_\pp}{M_\pp}$, where we have chosen in our program $m_\pp \dfeq 0.5$ and $M_\pp \dfeq 0.96$. The closer $M_\pp$ is to~$1$, the smaller the density we prove is required. However, larger~$M_\pp$ necessitates smaller~$\delta$ which increases the computation time. Through experimentation, we have chosen bounds, so that the program ran for a few days. Ultimately, with better computers (and more patience) one can improve our result. We note that experimentally we never came across any counterexample to our claims even outside of~$\cs$ (so long as the configuration satisfied the theoretical assumptions from Corollary~\ref{corollary:persistence-parameter-bounds}). We discuss this further in Section~\ref{section:discussion}.
		
		We can now calculate the upper bound on~$\alpha$ (the lower bound is just~$0$). For fixed~$\pp$ and~$\chi$ we claim that the case $\alpha \geq \arctan\big(\frac{\sqrt{\pp + \pp^2} \cos(\chi)}{1 + \pp \, \sin(\chi)}\big)$ is impossible. In this case the point $(0, 1) + \frac{\gp - (0, 1)}{\|\gp - (0, 1)\|}$ lies on the manifold, and is the closest to~$\gp$ among points on~$\mnf$. This is because its distance to~$\gp$ is bounded by~$\pp$ (by Lemma~\ref{lemma:distance-to-manifold}) which is smaller than $\rch = 1$, so its associated $\rch$-ball includes all points, closer to~$\gp$, and $\mnf$ cannot intersect an open associated $\rch$-ball --- see Figure~\ref{figure:program-projection-on-first-arc}.
		\begin{figure}[!ht]
			\centering
			\includegraphics[width=0.8\textwidth]{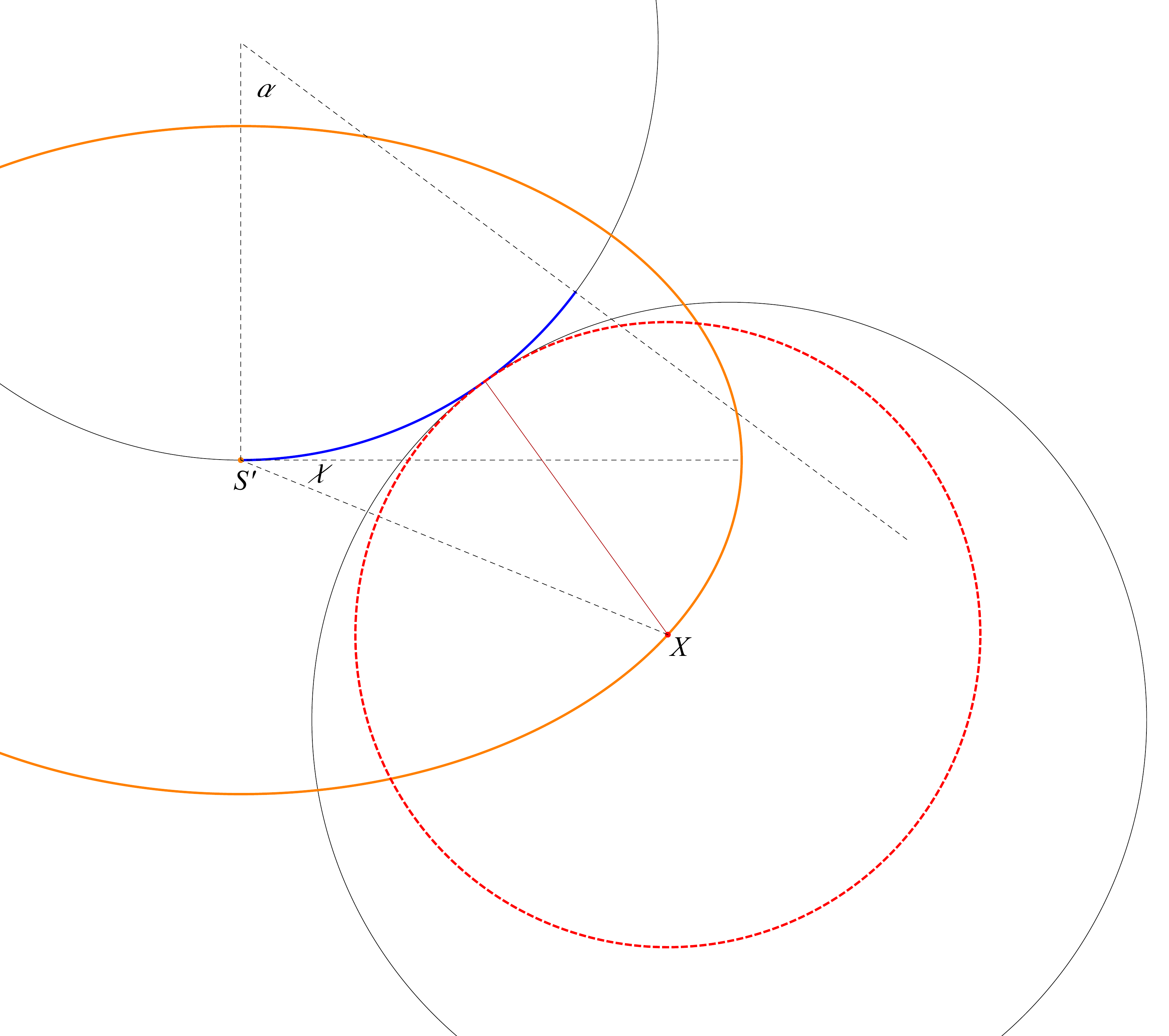}
			\caption{Too large~$\alpha$}\label{figure:program-projection-on-first-arc}
		\end{figure}
		We claim that the point $\pr(\gp) = (0, 1) + \frac{\gp - (0, 1)}{\|\gp - (0, 1)\|}$ lies in~$\of{\smpp'}$. This is a contradiction since then~$\sgp = \pr(\gp) \in \of{\smpp'}$.
		
		Clearly, it suffices to verify $\pr(\gp) \in \of{\smpp'}$ for $\chi = 0$ (for larger~$\chi$ the point~$\pr(\gp)$ lies on the $\rch$-arc further towards the ellipsoid center~$\smpp'$). If we put the coordinates of~$\pr(\gp)$ for $\chi = 0$ into the equation for the ellipsoid, we see that we need $\frac{2p^2 + p + 2 - 2 \sqrt{p^2+p+1}}{p^4 + p^3 + p^2} < 1$. This is equivalent to $-p^8-2 p^7+p^6+4 p^5+5 p^4+2 p^3-p^2 > 0$, which is equivalent to $-p^2 (p+1) \big(p^5+p^4-2 p^3-2 p^2-3 p+1\big) > 0$ which is further equivalent to $p^5+p^4-2 p^3-2 p^2-3 p+1 < 0$. The derivative of the polynomial on the left is
		\[5p^4 + 4p^3 - 6p^2 - 4p - 3 \leq -(5p^2 + 4p)(1-p^2) - 3 < 0,\]
		so $p^5+p^4-2 p^3-2 p^2-3 p+1$ is decreasing on $\intcc{m_\pp}{M_\pp} \subseteq \intoo{0}{1}$. The value of this polynomial at $m_\pp = 0.5$ is $-1.15625 < 0$, so the polynomial is negative on $\intcc{m_\pp}{M_\pp}$, as required.\footnote{We could reduce~$m_\pp$ to around~$0.273$, and the argument would still work. However, starting at~$0.5$ still leaves us with a decently sized length of the persistence interval, while (as we see later) reducing the Lipschitz constant for~$f$ and shortening the run-time of the program.}
		
		With this we have confirmed that it suffices to restrict ourselves to $\alpha \leq \arctan\big(\frac{\sqrt{\pp + \pp^2} \cos(\chi)}{1 + \pp \, \sin(\chi)}\big)$. As mentioned, this bound will be the largest at $\chi = 0$, so we will cover the relevant configurations for $\alpha \leq \arctan\big(\sqrt{\pp + \pp^2}\big)$, or equivalently (for $\alpha \in \intco{0}{\frac{\pi}{2}}$ and $\pp \in \intoo{0}{1}$) $\tan^2(\alpha) \leq \pp + \pp^2$, in particular $\tan^2(\alpha) \leq M_\pp + M_\pp^2$.
		
		Finally, we claim that we can restrict ourselves to $\sigma \in \intcc{0}{\pi}$. If the manifold were to trace a $\rch$-circle within a plane for longer than~$\pi$, it would necessarily be that $\rch$-circle. If it were to veer away from the circle, the medial axis would continue from the center of the circle to the area between the two parts of the manifold (see Figure~\ref{figure:program-arc-greater-than-pi}) which would contradict that the manifold's reach is~$\rch$.
		\begin{figure}[!htp]
			\centering
			\includegraphics[width=0.7\textwidth]{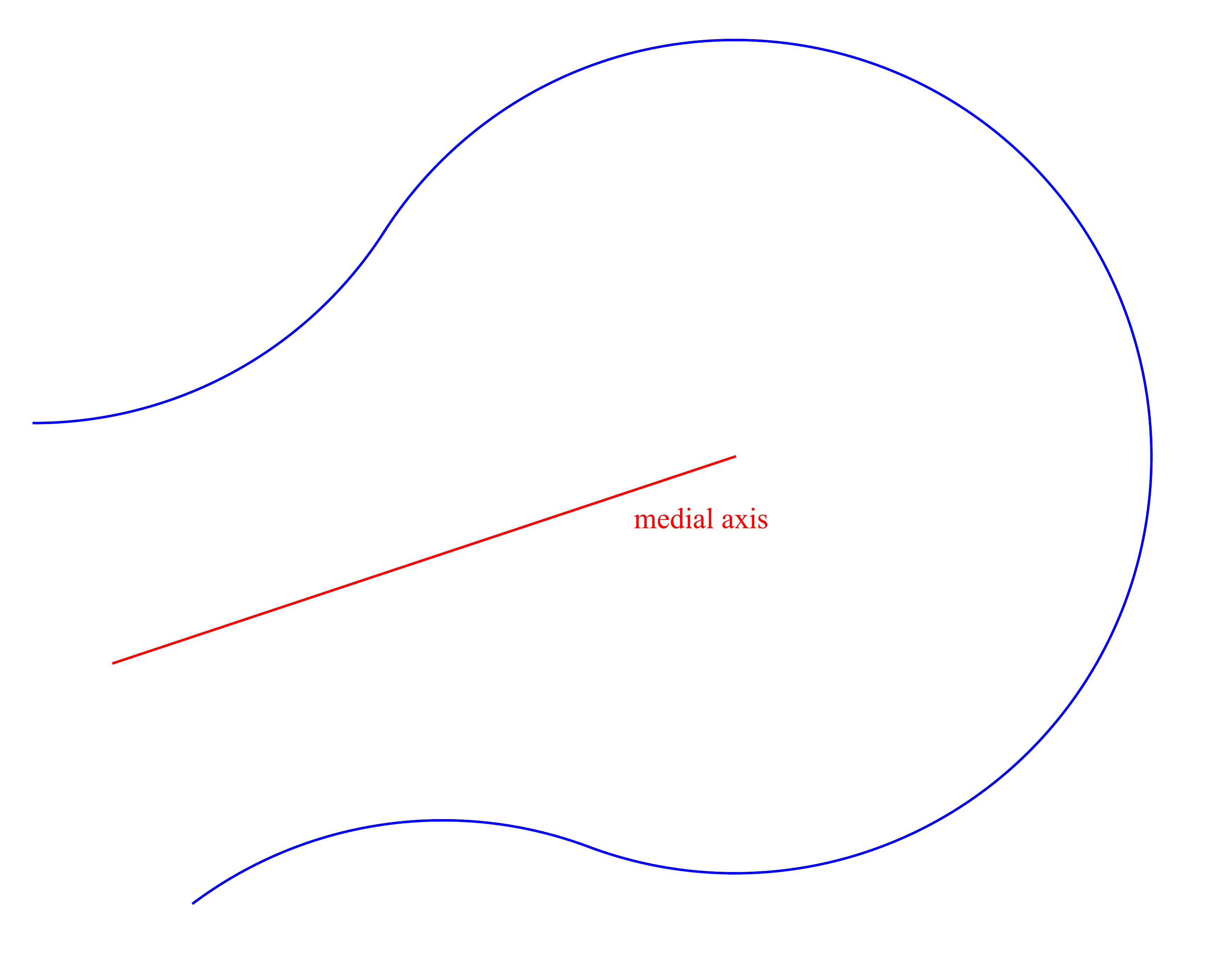}
			\caption{Medial axis of a manifold tracing an arc for longer than~$\pi$}\label{figure:program-arc-greater-than-pi}
		\end{figure}
		
		However, if the manifold was indeed just a circle in a plane, then $\sgp$ would be inside of~$\of{\smpp'}$ by the same argument we used when calculating the bound on~$\alpha$. Hence we may postulate $\sigma \in \intcc{0}{\pi}$.
		
		Having calculated the bounds on the variables, we may now define
		\[\cs \dfeq \set[2]{(\alpha, \sigma, \pp, \chi) \in \intcc{0}{\arctan(\sqrt{M_\pp + M_\pp^2})} \times \intcc{0}{\pi} \times \intcc{m_\pp}{M_\pp} \times \intcc{0}{\frac{\pi}{2}}}{\tan^2(\alpha) \leq \pp + \pp^2}.\]
		For the sake of a later calculation we also define a slightly bigger area,
		\[\ecs \dfeq \set[2]{(\alpha, \sigma, \pp, \chi) \in \intcc{0}{\arctan(\sqrt{M_\pp + M_\pp^2})} \times \intcc{0}{\pi} \times \intcc{m_\pp}{M_\pp} \times \intcc{0}{\frac{\pi}{2}}}{\tan^2(\alpha) \leq 2\pp + \pp^2}.\]
		Both $\cs$ and~$\ecs$ are $4$-dimensional rectangular cuboids with a small piece removed; in the $\alpha$-$\pp$-plane they look as shown in Figure~\ref{figure:program-configuration-regions}.
		\begin{figure}[!htp]
			\centering
			\includegraphics[width=0.55\textwidth]{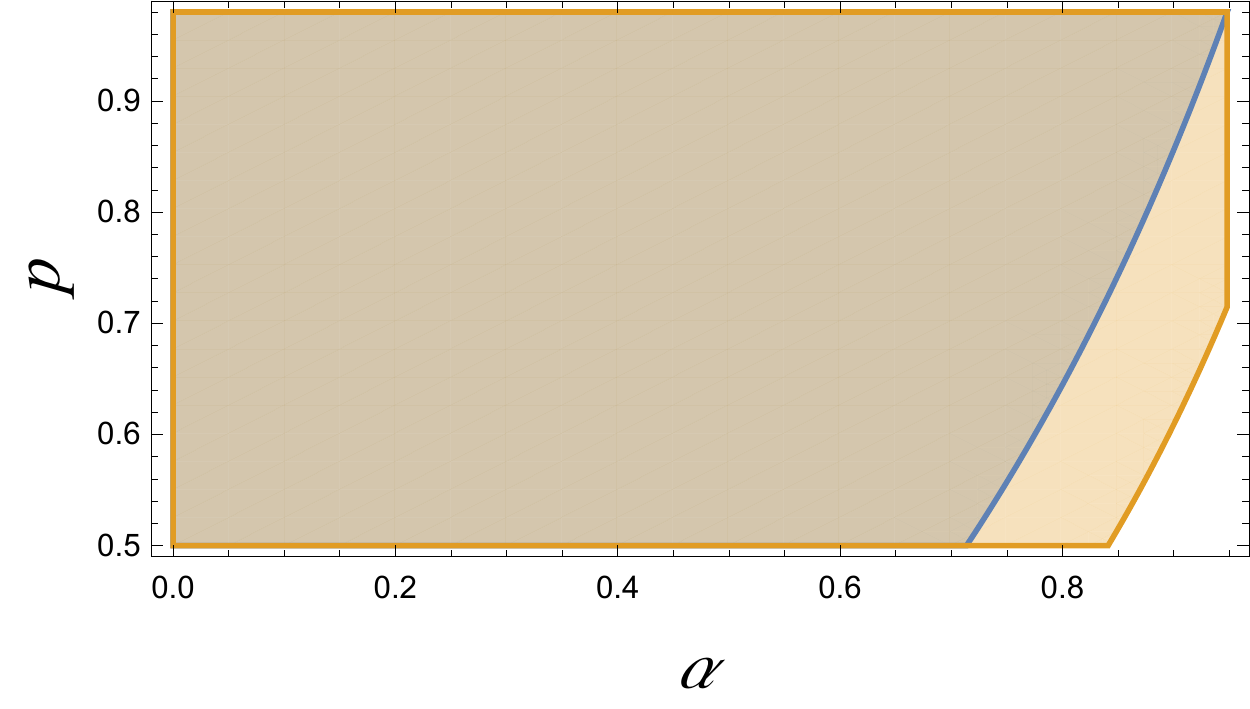}
			\caption{Regions~$\cs$ and~$\ecs$}\label{figure:program-configuration-regions}
		\end{figure}
		
		Given $(\alpha, \sigma, \pp, \chi) \in \cs$, we have $\gp = (\gp_T, \gp_N) = \big(\sqrt{\pp + \pp^2} \cos(\chi), -\pp \, \sin(\chi)\big)$. Let us denote the center of the $\rch$-ball, along the boundary of which lies the arc containing~$\smpp$, by~$\cab$. Observe from Figure~\ref{figure:program-point-coordinates} that $\cab = (0, 1) + 2 \big(\sin(\alpha), -\cos(\alpha)\big)$ and
		\[\smpp = (\smpp_T, \smpp_N) = \cab + \big(-\sin(\alpha-\sigma), \cos(\alpha-\sigma)\big) =\]
		\[= \big(2\sin(\alpha) - \sin(\alpha-\sigma), 1 - 2\cos(\alpha) + \cos(\alpha-\sigma)\big)\]
		(this works also if $\alpha-\sigma$ is negative).
		\begin{figure}[!ht]
			\centering
			\includegraphics[width=0.8\textwidth]{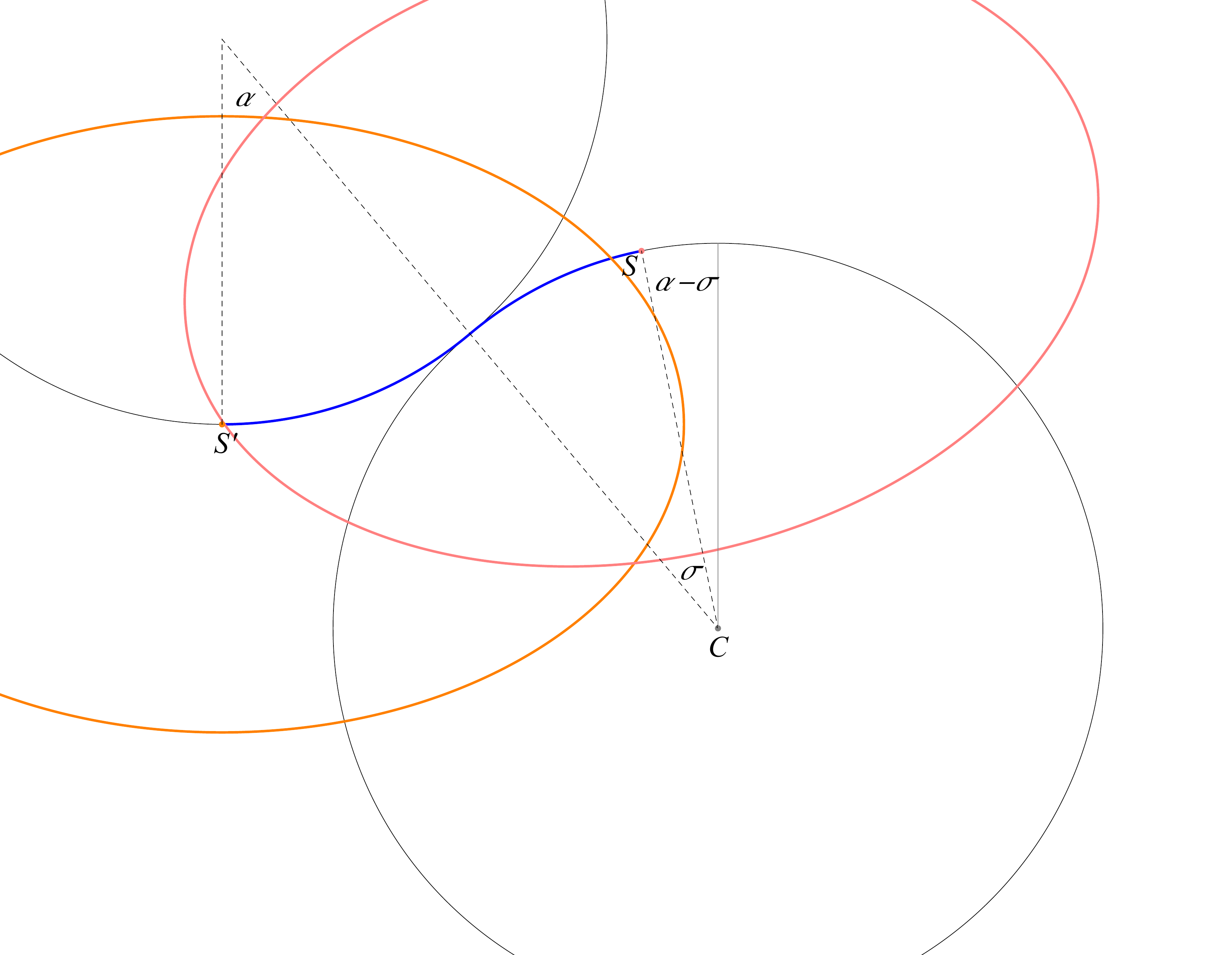}
			\caption{Position of~$\cab$ and~$\smpp$}\label{figure:program-point-coordinates}
		\end{figure}
		
		It will be convenient to define~$\pf$ on the larger area~$\ecs$ (although we are still only interested in positivity of~$\pf$ on~$\cs$). Recall that we want $\pf$ to be a function, so that its $0$-level set is the boundary of~$\of$, and is positive on $\of$ itself. Let $x, y$ be coordinates in our current coordinate system, $x', y'$ the coordinates in the coordinate system, translated by~$\smpp$, and $x'', y''$ the coordinates if we rotate the translated coordinate system by~$\alpha-\sigma$ in the positive direction. Hence
		\[\begin{bmatrix} x' \\ y' \end{bmatrix} = \begin{bmatrix} x \\ y \end{bmatrix} - \smpp, \qquad \begin{bmatrix} x'' \\ y'' \end{bmatrix} = \begin{bmatrix} \cos(\alpha-\sigma) & \sin(\alpha-\sigma) \\ -\sin(\alpha-\sigma) & \cos(\alpha-\sigma) \end{bmatrix} \cdot \begin{bmatrix} x' \\ y' \end{bmatrix}.\]
		In the rotated translated coordinate system, the equation for the boundary of the ellipse is $\frac{x''^2}{\pp + \pp^2} + \frac{y''^2}{\pp^2} = 1$, or equivalently $p^2 (\pp+1) - \big(x''^2 \pp + y''^2 (\pp+1)\big) = 0$. We therefore define $f\colon \ecs \to \RR$ by
		\[\pf(\alpha, \sigma, \pp, \chi) \dfeq \pp^2 (\pp+1) - \Big(\big(\cos(\alpha-\sigma) (\gp_T - \smpp_T) + \sin(\alpha-\sigma) (\gp_N - \smpp_N)\big)^2 \pp\]
		\[+ \big(-\sin(\alpha-\sigma) (\gp_T - \smpp_T) + \cos(\alpha-\sigma) (\gp_N - \smpp_N)\big)^2 (\pp+1)\Big).\]
		
		Recall that it follows from the multivariate Lagrange mean value theorem that for any $a, b \in \ecs$
		\[\big|\pf(a) - \pf(b)\big| \leq \max\|\nabla{\pf}\| \cdot \|a - b\|\]
		where the maximum of the norm of the gradient is taken over the line segment connecting the points $a$ and~$b$. In particular, the maximum over the entire~$\ecs$ is a Lipschitz coefficient for~$\pf$.
		
		This theorem holds for any pair of conjugate norms. We take the $\infty$-norm on~$\ecs$, and the \mbox{$1$-norm} for the gradient. The reason is that we cover the region~$\cs$ by cuboids which are almost cubes (in the centers of which we calculate the function values). The smaller the distance between the center of a cube and any of its points, the better the estimate we obtain. Hence
		\[\big|\pf(\alpha + \Delta{\alpha}, \sigma + \Delta{\sigma}, \pp + \Delta{\pp}, \chi + \Delta{\chi}) - \pf(\alpha, \sigma, \pp, \chi)\big| \leq\]
		\[\leq \max_{\ecs}\Big(\Big|\pd{\pf}{\alpha}\Big| + \Big|\pd{\pf}{\sigma}\Big| + \Big|\pd{\pf}{\pp}\Big| + \Big|\pd{\pf}{\chi}\Big|\Big) \cdot \max\set{|\Delta{\alpha}|, |\Delta{\sigma}|, |\Delta{\pp}|, |\Delta{\chi}|}.\]
		
		Before we estimate the absolute values of partial derivatives, let us make several preliminary calculations.
		
		First we put the function into a more convenient form.
		\begin{align*}
			\pf(\alpha, \sigma, \pp, \chi) &= \pp^2 (\pp+1) - \Big(\big(\cos(\alpha-\sigma) (\gp_T - \smpp_T) + \sin(\alpha-\sigma) (\gp_N - \smpp_N)\big)^2 \pp \\
			&\quad + \big(-\sin(\alpha-\sigma) (\gp_T - \smpp_T) + \cos(\alpha-\sigma) (\gp_N - \smpp_N)\big)^2 (\pp+1)\Big) = \\
			&= \pp^2 (\pp+1) - \Big(\big(\cos(\alpha-\sigma) (\gp_T - \smpp_T) + \sin(\alpha-\sigma) (\gp_N - \smpp_N)\big)^2 \\
			&\quad + \big(-\sin(\alpha-\sigma) (\gp_T - \smpp_T) + \cos(\alpha-\sigma) (\gp_N - \smpp_N)\big)^2\Big) \pp \\
			&\quad - \big(-\sin(\alpha-\sigma) (\gp_T - \smpp_T) + \cos(\alpha-\sigma) (\gp_N - \smpp_N)\big)^2 = \\
			&= \pp^2 (\pp+1) - \|\gp - \smpp\|^2 \pp \\
			&\quad - \big(-\sin(\alpha-\sigma) (\gp_T - \smpp_T) + \cos(\alpha-\sigma) (\gp_N - \smpp_N)\big)^2 = \\
			&= \pp^2 (\pp+1) - \spr{\gp - \smpp}{\gp - \smpp} \:\! \pp - \big(\spr{(-\sin(\alpha-\sigma), \cos(\alpha-\sigma))}{\gp - \smpp}\big)^2
		\end{align*}
		
		Now we calculate the bound on~$\gp - \smpp$ and its partial derivatives.
		\begin{align*}
			\gp - \smpp &= \Big(\sqrt{\pp + \pp^2} \cos(\chi) - 2\sin(\alpha) + \sin(\alpha-\sigma), -\pp \sin(\chi) - 1 + 2\cos(\alpha) - \cos(\alpha-\sigma)\Big) \\
			\|\gp - \smpp\| &\leq \|\gp - \cab\| + \|\cab - \smpp\| = \Big\|\big(\sqrt{\pp + \pp^2} \cos(\chi) - 2\sin(\alpha), -\pp \sin(\chi) - 1 + 2\cos(\alpha)\big)\Big\| + 1
		\end{align*}
		The norm will be the largest when either the components are largest ($\chi = 0$, $\alpha = 0$) or smallest ($\chi = \frac{\pi}{2}$, $\alpha = \alpha_{\max} \dfeq \arctan(\sqrt{2\pp + \pp^2})$). In the first case we get ${\|\gp - \cab\|^2 \leq \pp^2 + \pp + 1}$ and in the second (taking into account $\cos(\alpha_{\max}) = \frac{1}{\sqrt{1 + \tan^2(\alpha_{\max})}} = \frac{1}{\sqrt{1 + 2\pp + \pp^2}} = \frac{1}{1 + \pp}$)
		\begin{align*}
			\|\gp - \cab\|^2 &\leq 4\sin^2(\alpha_{\max}) + (-\pp - 1 + 2\cos(\alpha_{\max}))^2 \\
			&= 5 + \pp^2 + 2\pp - 4(1+\pp) \cos(\alpha_{\max}) \\
			&= 1 + \pp^2 + 2\pp \\
			&= (1+\pp)^2,
		\end{align*}
		so either way $\|\gp - \smpp\| \leq 2 + \pp \leq 2 + M_\pp$.
		
		\begin{align*}
			\Big\|\pd{(\gp - \smpp)}{\alpha}\Big\| &= \big\|\big(-2\cos(\alpha) + \cos(\alpha-\sigma), -2\sin(\alpha) + \sin(\alpha-\sigma)\big)\big\| \\
			&\leq 2 \:\! \big\|\big(-\cos(\alpha), -\sin(\alpha)\big)\big\| + \big\|\big(\cos(\alpha-\sigma), \sin(\alpha-\sigma)\big)\big\| \\
			&= 3 \\
			&\\
			\Big\|\pd{(\gp - \smpp)}{\sigma}\Big\| &= \big\|\big(-\cos(\alpha-\sigma), -\sin(\alpha-\sigma)\big)\big\| \\
			&= 1
		\end{align*}
		\begin{align*}
			&\\
			\Big\|\pd{(\gp - \smpp)}{\pp}\Big\| &= \Big\|\Big(\frac{1 + 2\pp}{2 \sqrt{\pp + \pp^2}} \cos(\chi), -\sin(\chi)\Big)\Big\| \\
			&= \Big\|\Big(\big(\frac{1 + 2\pp}{2 \sqrt{\pp + \pp^2}} - 1\big)\cos(\chi), 0\Big) + \Big(\cos(\chi), -\sin(\chi)\Big)\Big\| \\
			&\leq \Big|\frac{1 + 2\pp}{2 \sqrt{\pp + \pp^2}} - 1\Big| + 1 \\
			&= \frac{1 + 2\pp}{2 \sqrt{\pp + \pp^2}} \\
			&\leq \frac{1 + 2m_\pp}{2 \sqrt{m_\pp + m_\pp^2}}
		\end{align*}
		Here the last equality holds because $(1 + 2\pp)^2 = 1 + 4\pp + 4\pp^2 \geq 4\pp + 4\pp^2 = (2 \sqrt{\pp + \pp^2})^2$ and the last inequality holds because $\frac{1 + 2\pp}{2 \sqrt{\pp + \pp^2}}$ is a decreasing function: its derivative is $-\frac{1}{4 (\pp + \pp^2)^{\frac{3}{2}}}$.
		\begin{align*}
			\Big\|\pd{(\gp - \smpp)}{\chi}\Big\| &= \Big\|\big(-\sqrt{\pp + \pp^2} \sin(\chi), -\pp \cos(\chi)\big)\Big\| = \sqrt{\pp \sin^2(\chi) + \pp^2} \leq \sqrt{M_\pp + M_\pp^2}
		\end{align*}
		
		Next we calculate a bound on the term $\spr{(-\sin(\alpha-\sigma), \cos(\alpha-\sigma))}{\gp - \smpp}$ and its derivatives.
		\begin{align*}
			&\quad\ \big|\spr{\big(-\sin(\alpha-\sigma), \cos(\alpha-\sigma)\big)}{\gp - \smpp}\big| \leq \|\gp - \smpp\| \leq 2 + M_\pp \\
			&\\
			&\quad\ \Big|\pd{}{\alpha} \spr{\big(-\sin(\alpha-\sigma), \cos(\alpha-\sigma)\big)}{\gp - \smpp}\Big| \\
			&= \Big|\spr{\big(-\cos(\alpha-\sigma), -\sin(\alpha-\sigma)\big)}{\gp - \smpp} + \spr{\big(-\sin(\alpha-\sigma), \cos(\alpha-\sigma)\big)}{\pd{(\gp - \smpp)}{\alpha}}\Big| \\
			&\leq \|\gp - \smpp\| + \Big\|\pd{(\gp - \smpp)}{\alpha}\Big\| \leq (2 + M_\pp) + 3 = 5 + M_\pp \\
			&\\
			&\quad\ \Big|\pd{}{\sigma} \spr{\big(-\sin(\alpha-\sigma), \cos(\alpha-\sigma)\big)}{\gp - \smpp}\Big| \\
			&= \Big|\spr{\big(\cos(\alpha-\sigma), \sin(\alpha-\sigma)\big)}{\gp - \smpp} + \spr{\big(-\sin(\alpha-\sigma), \cos(\alpha-\sigma)\big)}{\pd{(\gp - \smpp)}{\sigma}}\Big| \\
			&\leq \|\gp - \smpp\| + \Big\|\pd{(\gp - \smpp)}{\sigma}\Big\| \leq (2 + M_\pp) + 1 = 3 + M_\pp \\
			&\\
			&\quad\ \Big|\pd{}{\pp} \spr{\big(-\sin(\alpha-\sigma), \cos(\alpha-\sigma)\big)}{\gp - \smpp}\Big| \\
			&= \Big|\spr{\big(-\sin(\alpha-\sigma), \cos(\alpha-\sigma)\big)}{\pd{(\gp - \smpp)}{\pp}}\Big| \leq \Big\|\pd{(\gp - \smpp)}{\pp}\Big\| \leq \frac{1 + 2m_\pp}{2 \sqrt{m_\pp + m_\pp^2}}
		\end{align*}
		\begin{align*}
			&\quad\ \Big|\pd{}{\chi} \spr{\big(-\sin(\alpha-\sigma), \cos(\alpha-\sigma)\big)}{\gp - \smpp}\Big| \\
			&= \Big|\spr{\big(-\sin(\alpha-\sigma), \cos(\alpha-\sigma)\big)}{\pd{(\gp - \smpp)}{\chi}}\Big| \leq \Big\|\pd{(\gp - \smpp)}{\chi}\Big\| \leq \sqrt{M_\pp + M_\pp^2}
		\end{align*}
		
		We can now estimate the partial derivatives of~$\pf$.
		\begin{align*}
			\Big|\pd{f}{\alpha}\Big| &\leq 6 M_\pp (2 + M_\pp) + 2 (2 + M_\pp) (5 + M_\pp) = 20 + 26 M_\pp + 8 M_\pp^2 \\
			&\\
			\Big|\pd{f}{\sigma}\Big| &\leq 2 (2 + M_\pp) M_\pp + 2 (2 + M_\pp) (3 + M_\pp) = 12 + 14 M_\pp + 4 M_\pp^2 \\
			&\\
			\Big|\pd{f}{\pp}\Big| &\leq 3 M_\pp^2 + 2 M_\pp + 2 (2 + M_\pp) \frac{1 + 2m_\pp}{2 \sqrt{m_\pp + m_\pp^2}} M_\pp + (2 + M_\pp)^2 + 2 (2 + M_\pp) \frac{1 + 2m_\pp}{2 \sqrt{m_\pp + m_\pp^2}} \\
			&= 4 + 6 M_\pp + 4 M_\pp^2 + 2 (2 + M_\pp) (1 + M_\pp) \frac{1 + 2m_\pp}{2 \sqrt{m_\pp + m_\pp^2}} \\
			&\\
			\Big|\pd{f}{\chi}\Big| &\leq 2 (2 + M_\pp) \sqrt{M_\pp + M_\pp^2} \, M_\pp + 2 (2 + M_\pp) \sqrt{M_\pp + M_\pp^2} \\
			&= 2 (2 + M_\pp) (1 + M_\pp) \sqrt{M_\pp + M_\pp^2}
		\end{align*}
		
		Hence a Lipschitz coefficient for~$\pf$ on~$\ecs$ is
		\begin{multline*}
			\lc \dfeq 20 + 26 M_\pp + 8 M_\pp^2 + 12 + 14 M_\pp + 4 M_\pp^2 + 4 + 6 M_\pp + 4 M_\pp^2 \\
			\qquad + 2 (2 + M_\pp) (1 + M_\pp) \frac{1 + 2m_\pp}{2 \sqrt{m_\pp + m_\pp^2}} + 2 (2 + M_\pp) (1 + M_\pp) \sqrt{M_\pp + M_\pp^2}
		\end{multline*}
		\[= 36 + 46 M_\pp + 16 M_\pp^2 + 2 (2 + M_\pp) (1 + M_\pp) \Big(\frac{1 + 2m_\pp}{2 \sqrt{m_\pp + m_\pp^2}} + \sqrt{M_\pp + M_\pp^2}\Big)\]
		which is a little less than~$125$.
		
		The idea behind the program is that it accepts a value~$\delta \in \RR_{> 0}$, sets each of the variables $\alpha$, $\sigma$, $\pp$, $\chi$ at~$\delta$ away from the edge of~$\ecs$ and calculates the values of~$\pf$ in a lattice of points, of which any two consecutive ones differ in the values of the variables by~$2\delta$. The idea is that $\infty$-balls (cubes) with the centers in the lattice points and radius~$\delta$ cover~$\cs$, so if the values of~$\pf$ in these points is~$> \lc \, \delta$, then $\pf$ is positive.
		
		This requires two remarks, however. First, if one tries to evenly cover a cuboid by cubes with edge length~$2\delta$ and with centers within the cuboid, the cuboid will not be covered, if dividing any cuboid edge length by~$2\delta$ yields a remainder, greater than~$\delta$. For this reason, in the program we decrease each cube edge length slightly (by reducing the step of each variable) so that the now slightly distorted cubes exactly cover the cuboid enclosing~$\cs$ and~$\ecs$ if we take their centers from the lattice spanning the cuboid (though since we are trying to only cover~$\cs$, we do not need to take these centers from the entire cuboid).
		
		The second problem is that~$\cs$ is not actually a cuboid and might not get covered by the distorted cubes if we only took those with the centers in~$\cs$. However, we claim that the distorted cubes cover~$\cs$ if we take the centers from~$\ecs$, as long as $\delta$ is small enough.
		
		Recall Figure~\ref{figure:program-configuration-regions}; since the dependence of the lower bound for~$\pp$ on~$\alpha$ is increasing for both~$\cs$ and~$\ecs$, it suffices to check that if $\big(\alpha, \sigma, \pp, \chi\big) \in \cs$, then $\big(\min\set{\alpha + \delta, \arctan(\sqrt{M_\pp + M_\pp^2})}, \sigma, \max\set{\pp - \delta, m_\pp}, \chi\big) \in \ecs$.
		
		We have
		\[\arctan(\sqrt{M_\pp + M_\pp^2}) \leq \arctan(\sqrt{2}),\]
		\[\big|(\tan^2(\alpha))'\big| = \Big|\frac{2 \tan(\alpha)}{\cos^2(\alpha)}\Big| = \big|2 \tan(\alpha) (1 + \tan^2(\alpha))\big| \leq 6 \sqrt{2} \leq 9.\]
		Hence $\tan^2(\alpha)$ has a Lipschitz coefficient of~$9$ on the relevant region. Similarly, $\pp^2$ has a Lipschitz coefficient of~$2$.
		
		Then
		\[\tan^2(\alpha + \delta) \leq \tan^2(\alpha) + 9\delta \leq \pp + \pp^2 + 9\delta \leq 2(\pp - \delta) + (\pp - \delta)^2 - \pp + 2\delta + 2\delta + 9\delta =\]
		\[= 2(\pp - \delta) + (\pp - \delta)^2 - \pp + 13\delta \leq 2(\pp - \delta) + (\pp - \delta)^2\]
		for $\delta \in \intoc{0}{\frac{m_\pp}{13}}$. In particular, $\delta \in \intoc{0}{0.01}$ suffices, also in the cases where we hit the edges at~$\arctan(\sqrt{M_\pp + M_\pp^2})$ and/or~$m_\pp$ (we did not have to be too picky about these particular estimates; the actual value of~$\delta$ we run the program with is far smaller, at~$0.0004$, as we explain below).
		
		There is one more issue which prevents us from getting as nice of a result with a computer program as we would get with a theoretical derivation. Recall that we require $\hd < \sqrt{2\pp \big(\sqrt{\rch (\pp + 2\rch)} - \rch\big)} = \sqrt{2\pp \big(\sqrt{\pp + 2} - 1\big)}$. The program verifies that if $\sgp \notin \of{\smpp'}$, then $\gp \in \of$, where $\smpp$ is chosen within $\hd$-distance from~$\sgp$, hence the entire line segment from~$\gp$ to~$\sgp$ is in~$\of$. However, if we allow $\hd$ to get arbitrarily close to~$\sqrt{2\pp \big(\sqrt{\pp + 2} - 1\big)}$, then the value of~$\pf$ gets arbitrarily close to zero, and we cannot use our method to prove that~$\pf$ is positive. To avoid this, we decrease the upper bound on the distance between $\smpp$ and~$\sgp$ to $\sqrt{2\pp \big(\sqrt{\pp + 2} - 1\big) - \hd_\mathrm{off}}$ for $\hd_\mathrm{off} = 0.55$ (we chose this value experimentally, so that the result of the program was sufficiently good).
		
		After some experimentation, we ran the program with $\delta = 0.0004$. The resulting smallest value of~$\pf$ that the program returned, was~$0.068546$.
		
		Recall that~$\pf$ has a Lipschitz coefficient of~$125$. Since any possible configuration is at most~${\delta = 0.0004}$ away from some point in the lattice where the program calculates~$\pf$, the values that~$\pf$ can take are at most $125 \cdot 0.0004 = 0.05$ smaller than the values, calculated by the program. In particular, $\pf$ is necessarily positive.
		
		The source code of our c{\small++} program is available at~\url{https://people.math.ethz.ch/~skalisnik/ellipsoids.cpp}.
		
		The price of this method is that we had to decrease the size of the theoretical interval for the persistence parameter $\pp \in \intoo{\lppb}{\rch}$ which in particular requires greater density sample for the proof than is strictly necessary. We discuss this in Section~\ref{section:discussion}.
		
		Let us summarize the results we have obtained in this section. We have seen that if a point~$\gp \in \cn$ is in at least two of the closed ellipsoids, then there exists $\smpp \in \smp$ such that $\gp \in \cf$ and $\pr(\gp) \in \of$. This happened in one of two ways. The first closed ellipsoid we took~$\gp$ from could already satisfy this property, or we could find an ellipsoid with the center close to~$\pr(\gp)$ which contained both $\gp$ and~$\pr(\gp)$ in its interior. If we start with $\gp \in \on$ though, we can pick as the first ellipsoid one that has $\gp$ in its interior, which means that we can always conclude the statement of Lemma~\ref{lemma:program-conclusion}.

	\section{Construction of the Deformation Retraction}\label{section:deformation-retraction-construction}
	
		In this section we show that under the same assumptions on~$\rch$ and~$\pp$ as in the previous section, the union of the open ellipsoids around sample points deformation retracts onto the manifold~$\mnf$.
		
		Informally, the idea of the deformation retraction is as follows. For a point~$\gp$ in an open ellipsoid~$\of$, consider the closed ellipsoid~$\cf{\smpp}{q}$ where $q \dfeq \pep(\gp)$ (Definition~\ref{definition:depth-parameter}), the boundary of which contains~$\gp$. If the vector $\prv(\gp)$ points into the interior of this closed ellipsoid, we move in the direction of~$\prv(\gp)$, i.e.~we use the normal deformation retraction. Otherwise, we move in the direction of the projection of the vector $\prv(\gp)$ onto the tangent space $\ts[\bd{\cf{\smpp}{q}}]{\gp}$. This causes us to slide along the boundary~$\bd{\cf{\smpp}{q}}$. Either way, we remain within $\cf{\smpp}{q}$ (and therefore within~$\on$) and eventually reach the manifold~$\mnf$. This procedure is problematic for points which are in more than one ellipsoid, but we can glue together the directions of the deformation retraction with a suitable partition of unity. Figure~\ref{figure:deformation-retraction-sketch} illustrates this idea.
		\begin{figure}[!htp]
			\centering
			\includegraphics[width=0.99\textwidth]{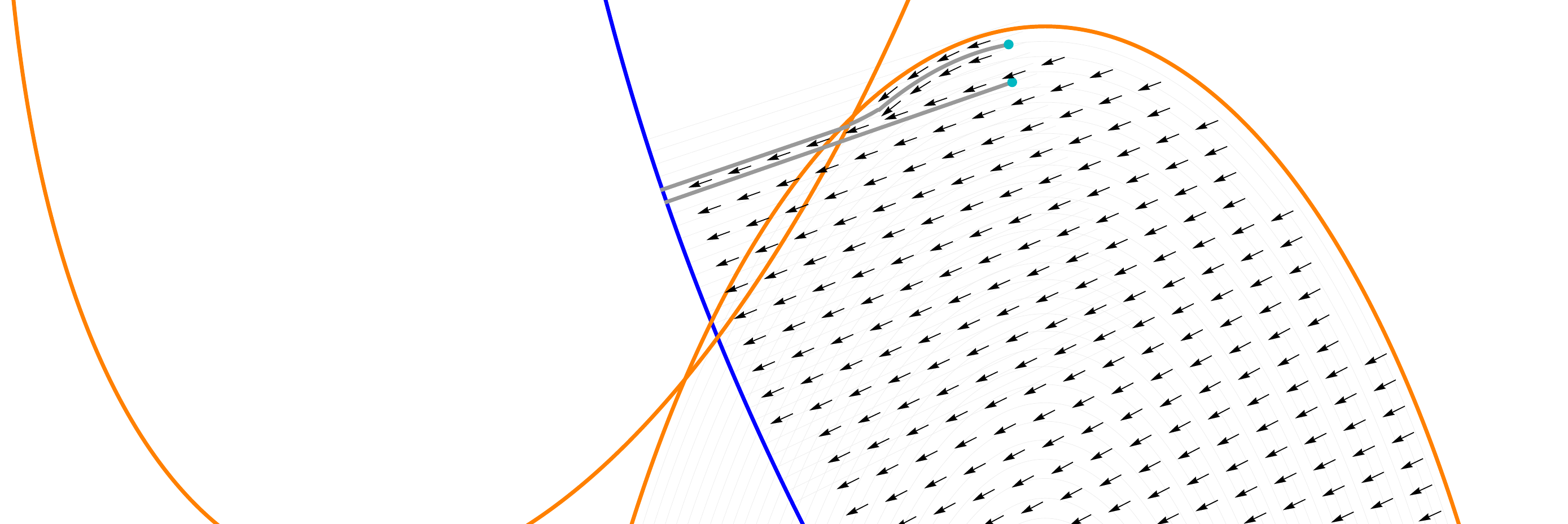}
			\caption{Idea for the deformation retraction}\label{figure:deformation-retraction-sketch}
		\end{figure}
		
		To make this work, we will need precise control over the partition of unity, which is the topic of Subsection~\ref{subsection:partition-of-unity}. Then in Subsection~\ref{subsection:vector-field} we define the vector field which gives directions, in which we deformation retract. Subsection~\ref{subsection:flow} proves that the flow of this vector field has desired properties. We then use this flow to explicitly give the definition of the requisite deformation retraction in Subsection~\ref{subsection:deformation-retraction}.
		
		\subsection{The Partition of Unity}\label{subsection:partition-of-unity}
		
			For each $\smpp \in \smp$ define
			\[\se \dfeq \on \setminus \bigcup_{\smpp' \in \smp \setminus \set{\smpp}} \of{\smpp'}, \qquad\qquad \de \dfeq \on \setminus \of.\]
			The sets $\se$ and~$\de$ are closed in~$\on$ because they are complements within~$\on$ of open sets. Note that $\se \subseteq \of = \on \setminus \de$. In particular, $\se$ and~$\de$ are disjoint.
			
			\begin{proposition}
				If $\smpp', \smpp'' \in \smp$ and $\smpp' \neq \smpp''$, then $\se[\smpp'] \subseteq \de[\smpp'']$ and $\se[\smpp'] \cap \se[\smpp''] = \emptyset$.
			\end{proposition}
			
			\begin{proof}
				For any $\gp \in \on$, if $\gp \in \se[\smpp']$, then $\gp \notin \of{\smpp''}$, so $\gp \in \de[\smpp'']$. Consequently $\se[\smpp'] \cap \se[\smpp''] \subseteq \de[\smpp''] \cap \se[\smpp''] = \emptyset$.
			\end{proof}
			
			The only way $\de$ could be empty is if the sample~$\smp$ is a singleton which can only happen when $\mnf$ is a singleton, but this possibility is excluded by the assumption that the dimension of~$\mnf$ is positive. The distance to any non-empty set is a well-defined real-valued function, the zeroes of which form the closure of the set.
			
			Define
			\begin{align*}
				\tse &\dfeq \begin{cases} \set{\gp \in \on}{\dst(\se, \gp) \leq \tfrac{1}{2} \dst(\de, \gp)} & \text{if~$\se \neq \emptyset$,} \\ \emptyset & \text{if~$\se = \emptyset$,} \end{cases} \\
				\tde &\dfeq \begin{cases} \set{\gp \in \on}{\dst(\de, \gp) \leq \tfrac{3}{2} \dst(\se, \gp)} & \text{if~$\se \neq \emptyset$,} \\ \on & \text{if~$\se = \emptyset$.} \end{cases}
			\end{align*}
			
			The sets $\tse$ and~$\tde$ are disjoint. If we had $\gp \in \tse \cap \tde$, then $\dst(\se, \gp) \leq \tfrac{1}{2} \dst(\de, \gp) \leq \tfrac{3}{4} \dst(\se, \gp)$, so $\dst(\se, \gp) = \dst(\de, \gp) = 0$, meaning $\gp \in \se \cap \de$, a contradiction. Note also that $\de \subseteq \tde$ and
			\[\se \ \subseteq \ \tse \ \subseteq \ \on \setminus \tde \ \subseteq \ \on \setminus \de \ = \ \of.\]
			
			The sets $\tse$ and~$\tde$ are closed in~$\on$ because they are (empty or) preimages of~$\RR_{\geq 0}$ under continuous maps ${\gp \mapsto \tfrac{1}{2} \dst(\de, \gp) - \dst(\se, \gp)}$ and ${\gp \mapsto \tfrac{3}{2} \dst(\se, \gp) - \dst(\de, \gp)}$. Using the smooth version of Urysohn's lemma~\cite{lee2013smooth}, choose a smooth function $\pu\colon \on \to \intcc{0}{1}$ such that $\pu$ is constantly~$1$ on~$\tse$ and constantly~$0$ on~$\tde$.
			
			Recall that the \df{support} $\supp{f}$ of a continuous real-valued function~$f$ is defined as the closure of the complement of the zero set, where both the complementation and the closure are calculated in the domain of~$f$.
			
			\begin{proposition}
				For every $\smpp \in \smp$ and $\gp \in \on$, if $\gp \in \supp{\pu}$, then $\on \setminus \tde, \se \neq \emptyset$, $\dst(\tde, \gp) \geq \tfrac{3}{2} \dst(\se, \gp)$ and $\gp \in \of$.
			\end{proposition}
			
			\begin{proof}
				Since $\gp \in \supp{\pu}$, the support of~$\pu$ is non-empty, so $\tde \neq \on$, therefore $\se \neq \emptyset$. The set $\set[1]{\sgp \in \on}{\dst(\tde, \gp) \geq \tfrac{3}{2} \dst(\se, \gp)}$ is closed in~$\on$ and contains~$\pu^{-1}(\intoc{0}{1})$, so it contains~$\supp{\pu}$.
				
				If $\gp \in \se$, then $\gp \in \of$. If $\gp \notin \se$, then $\dst(\de, \gp) \geq \dst(\tde, \gp) \geq \tfrac{3}{2} \dst(\se, \gp) > 0$, so again $\gp \in \of$.
			\end{proof}
			
			\begin{proposition}\label{proposition:supports-pairwise-disjoint}
				The supports\footnote{Recall that a function~$\pu$ is defined on~$\on$, so when calculating its support, we take the closure within~$\on$. This support is in general not closed in~$\RR^\ad$.} of functions~$\pu$ are pairwise disjoint. Hence every point in~$\on$ has a neighbourhood which intersects the support of at most one~$\pu$.
			\end{proposition}
			
			\begin{proof}
				Take $\smpp', \smpp'' \in \smp$, $\smpp' \neq \smpp''$ and let $\gp \in \pu[\smpp'] \cap \pu[\smpp'']$. Then $\dst(\se[\smpp'], \gp) \leq \tfrac{2}{3} \dst(\de[\smpp'], \gp) \leq \tfrac{2}{3} \dst(\se[\smpp''], \gp)$ and likewise $\dst(\se[\smpp''], \gp) \leq \tfrac{2}{3} \dst(\se[\smpp'], \gp)$, implying $\dst(\se[\smpp'], \gp) = \dst(\se[\smpp''], \gp) = 0$, so $\gp \in \se[\smpp'] \cap \se[\smpp'']$, a contradiction.
				
				Since $\supp{\pu} \subseteq \of$ for all~$\smpp \in \smp$, the family of supports of~$\pu$ is also locally finite. Thus any $\gp \in \on$ has a neighbourhood $\st{U} \subseteq \on$ which intersects only finitely many supports, at most one of which contains~$\gp$. The intersection of the complements of the rest with the set~$\st{U}$ is a neighbourhood of~$\gp$ which intersects at most one support.
			\end{proof}
			
			From these results we can conclude that $\gp \mapsto \sum_{\smpp \in \smp} \pu(\gp)$ gives a well-defined smooth map $\on \to \intcc{0}{1}$. We may therefore define a smooth map $\pun\colon \of \to \intcc{0}{1}$,
			\[\pun(\gp) \dfeq 1 - \sum_{\smpp \in \smp} \pu(\gp).\]
			Thus the family of maps~$\pu$, $\smpp \in \smp$, together with~$\pun$, forms a smooth partition of unity on~$\on$.
			
			We will need two more subsets of~$\on$:
			\begin{align*}
				\sea &\dfeq \set[1]{\gp \in \on}{\some{\smpp \in \smp}{\se \neq \emptyset \land \dst(\se, \gp) < \tfrac{1}{2} \dst(\de, \gp)}}, \\
				\npa &\dfeq \set[1]{\gp \in \on}{\some{\smpp \in \smp}{\gp \in \of \land \pr(\gp) \in \of}}.
			\end{align*}
			
			\begin{lemma}\label{lemma:open-cover-of-ellipsoid-area}
				The sets $\sea$ and~$\npa$ are open in~$\on$ and in~$\RR^\ad$, and $\sea \cup \npa = \on$.
			\end{lemma}
			
			\begin{proof}
				The given sets are open in~$\on$ since $\sea = \bigcup_{\smpp \in \smp, \se \neq \emptyset} \big(\gp \mapsto \tfrac{1}{2} \dst(\de, \gp) - \dst(\se, \gp)\big)^{-1}(\RR_{> 0})$ and $\npa = \bigcup_{\smpp \in \smp} \of \cap \pr^{-1}\big(\of\big)$. As $\on$ is open in~$\RR^\ad$, they are also open in~$\RR^\ad$.
				
				Assume that $\gp \in \on \setminus \sea$. If $\gp$ was in any~$\se$, we would have $0 = \dst(\se, \gp) \geq \tfrac{1}{2} \dst(\de, \gp)$, so $\gp \in \se \cap \de$, a contradiction. Since $\gp$ is in no~$\se$, it must be in at least two~$\of$, so $\gp \in \npa$ by Lemma~\ref{lemma:program-conclusion}.
			\end{proof}
		
		\subsection{The Velocity Vector Field}\label{subsection:vector-field}
		
			Let us define for each $\smpp \in \smp$ the vector field $\lavf\colon \of \to \RR^\ad$ as follows. Given $\gp \in \of$, let $\hs$ denote the $\ad$-dimensional closed half-space which is bounded by the hyperplane~$\ts[\bd{\cf{\smpp}{\pep(\gp)}}]{\gp}$ and which contains $\cf{\smpp}{\pep(\gp)}$. Define $\lavf(\gp)$ to be the projection of the vector $\prv(\gp)$ to the closest point in~$\hs$. Explicitly, if we introduce any orthonormal coordinate system with the origin in~$\gp$ such that the last coordinate axis points orthogonally to~$\bd{\cf{\smpp}{\pep(\gp)}}$ into the interior of~$\cf{\smpp}{\pep(\gp)}$, then the projection in these coordinates is given by $(x_1, \ldots, x_{n-1}, x_n) \mapsto (x_1, \ldots, x_{n-1}, \max\set{x_n, 0})$.
			
			\begin{proposition}\label{proposition:local-vector-field-lipschitz}
				The vector field $\lavf\colon \of \to \RR^\ad$ is Lipschitz with a bound on a Lipschitz coefficient independent from~$\smpp$.
			\end{proposition}
			
			\begin{proof}
				The projection onto a half-space is $1$-Lipschitz. By setting $\rch = 1$ and $r = \pp$ in Lemma~\ref{lemma:continuity-of-closest-point-projection}, we see that the map~$\prv$ is $(\frac{1}{1-\pp} + 1)$-Lipschitz on $\of \subseteq \ctnm{1-\pp}$. As the composition of these two maps, the vector field~$\lavf$ is Lipschitz with the product Lipschitz coefficient, i.e.~also~$\frac{1}{1-\pp} + 1$.
			\end{proof}
			
			For any $\smpp\in \smp$ and $\gp \in \of \setminus{\mnf}$ let $\pa$ denote the angle between the vectors $\prv(\gp)$ and~$\lavf(\gp)$, and let $\nhl$ denote the closed half-line which starts at~$\gp$, is orthogonal to~$\bd\cf{\smpp}{\pep(\gp)}$ and points into the exterior of~$\cf{\smpp}{\pep(\gp)}$.
			
			\begin{lemma}\label{lemma:normal-projection-away-from-normal-space-of-ellipsoid}
				Let $\smpp \in \smp$ and $\gp \in \of \setminus{\mnf}$. Then $\pr(\gp) \notin \nhl$; in fact, the angle between $\prv(\gp)$ and~$\nhl$ is bounded from below by~$\arccot(\tfrac{1}{\sqrt{2}})$. Hence $0 \leq \pa \leq \arccos\big(\sqrt{\frac{2}{3}}\big)$; in particular $\cos(\pa) \geq \sqrt{\frac{2}{3}}$.
			\end{lemma}
			
			\begin{proof}
				Let $q \dfeq \pep(\gp)$. We have $q \leq \pp < 1$ and since $\gp \notin \mnf$, in particular $\gp \neq \smpp$, we have $q > 0$. Let $\vec{n}$ be the unit vector, orthogonal to the boundary of~$\cf{\smpp}{q}$ and pointing into the exterior of~$\cf{\smpp}{q}$, so that $\nhl = \set{X + t \vec{n}}{t \in \RR_{\geq 0}}$. Let $\ell \dfeq \|\prv(\gp)\|$; by assumption $\gp \notin \mnf$, so $\ell > 0$, and we may define $\vec{m} \dfeq \frac{\prv(\gp)}{\ell}$. Also, since $\on \subseteq \otnm{\rch} = \otnm{1}$, we have $\ell < 1$.
				
				Use Lemma~\ref{lemma:general-manifold-properties} to introduce a planar tangent-normal coordinate system with the origin at~$\smpp$ which contains~$\gp$ as well as~$\vec{n}$, hence the whole~$\nhl$. Without loss of generality assume that $\gp$ lies in the closed first quadrant, so that we have $\chi \in \intcc{0}{\frac{\pi}{2}}$ with $\gp = \big(\sqrt{q+q^2} \sin(\chi), q \cos(\chi)\big)$ (the angle is measured from~$\ns$).
				
				Let us first prove that $\pr(\gp) \notin \nhl$. Assume to the contrary that this were the case, so that $\vec{m} = \vec{n}$. We will derive the contradiction by showing that the open $\rch$-ball with the center in~$\gp - (1-\ell) \vec{m}$, associated to~$\mnf$ at~$\pr(\gp)$, intersects all open $\rch$-balls, associated to~$\mnf$ at~$\smpp$. Two of those have their centers in the tangent-normal plane we are considering, and necessarily one of those is the $\rch$-ball at~$\smpp$ which is the furthest away from the $\rch$-ball with the center in~$\gp - (1-\ell) \vec{m}$. It thus suffices to check that the latter intersects the former two.
				
				First we explicitly calculate~$\vec{m}$.
				\[\vec{m} = \vec{n} = \frac{\big(q \sin(\chi), \sqrt{q+q^2} \cos(\chi)\big)}{\big\|\big(q \sin(\chi), \sqrt{q+q^2} \cos(\chi)\big)\big\|} = \frac{\big(q \sin(\chi), \sqrt{q+q^2} \cos(\chi)\big)}{\sqrt{q \cos^2(\chi) + q^2}}\]
				
				We derive the contradiction by showing that $\dst\big(\gp - (1-\ell) \vec{m}, (0, \pm{1})\big) < 2$.
				\begin{align*}
					&\dst\big(\gp - (1-\ell) \vec{m}, (0, \pm{1})\big)^2 \\
					&= \big\|\gp - (1-\ell) \vec{m} - (0, \pm{1})\big\|^2 \\
					&= \|\gp\|^2 + (1-\ell)^2 + 1 - 2(1-\ell)\spr{\gp}{\vec{m}} - 2\spr{\gp - (1-\ell) \vec{m}}{(0, \pm{1})} \\
					&= q \sin^2(\chi) + q^2 + (1-\ell)^2 + 1 - 2(1-\ell) \tfrac{q \sqrt{q+q^2}}{\sqrt{q \cos^2(\chi) + q^2}} \\
					&\qquad\qquad \mp 2 \Big(q \cos(\chi) - (1-\ell) \tfrac{\sqrt{q+q^2} \cos(\chi)}{\sqrt{q \cos^2(\chi) + q^2}}\Big) \\
					&= q \big(1 - \cos^2(\chi)\big) + q^2 + (1-\ell)^2 + 1 - 2(1-\ell) \tfrac{q \sqrt{1+q}}{\sqrt{\cos^2(\chi) + q}} \\
					&\qquad\qquad \mp 2 \cos(\chi) \Big(q - (1-\ell) \tfrac{\sqrt{1+q}}{\sqrt{\cos^2(\chi) + q}}\Big) \\
					&= q \big(2 - \big(1 \pm \cos(\chi)\big)^2\big) + q^2 + (1-\ell)^2 + 1 - 2(1-\ell) \sqrt{\tfrac{1+q}{\cos^2(\chi) + q}} \big(q \mp \cos(\chi)\big) \\
					&= (1 + q)^2 - q \big(1 \pm \cos(\chi)\big)^2 + (1-\ell)^2 - 2(1-\ell) \sqrt{\tfrac{1+q}{\cos^2(\chi) + q}} \big(q \mp \cos(\chi)\big) \\
					&\leq (1 + q)^2 - q \big(1 \pm \cos(\chi)\big)^2 + 1 - 2(1-\ell) \sqrt{\tfrac{1+q}{\cos^2(\chi) + q}} \big(q \mp \cos(\chi)\big)
				\end{align*}
				We verify that this expression is~$< 4$ for $q, \ell \in \intoo{0}{1}$ and $\chi \in \intcc{0}{\frac{\pi}{2}}$ with the help from \textit{Mathematica}, see file \texttt{ProjectionNotOnHalfline.nb}, available at~\url{https://people.math.ethz.ch/~skalisnik/ProjectionNotOnHalfline.nb}.
				
				This has shown that $\vec{m}$ cannot be equal to~$\vec{n}$ because in that case the open $\rch$-ball with the center in $\gp - (1-\ell) \vec{m}$ would intersect all open $\rch$-balls, associated to~$\smpp$. A lower bound on the angle between $\vec{m}$ and~$\vec{n}$ is therefore the minimal angle, by which we must deviate from~$\vec{n}$, so that we no longer have an intersection of the aforementioned balls.
				
				Observe that if two balls intersect, the closer their centers are, the greater the angle we must turn one of them by around a point on its boundary, so that they stop intersecting. Hence, if we try to turn the ball with the center in $\gp - (1-\ell) \vec{n}$ around the point~$\pr(\gp)$ so that it no longer intersects all balls, associated to~$\smpp$, we can get a lower bound on the angle by turning it by a minimal angle so that it no longer intersects the ball, associated to~$\smpp$, which is furthest away. The center of this furthest ball lies in our planar tangent-normal coordinate system in which it has coordinates~$(0, -1)$. Furthermore, the greater the $\ell$ is, the further $\gp - (1-\ell) \vec{n}$ is away from~$(0, -1)$. The minimal angle by which we must turn the ball with this center continuously depends on~$\ell$ and can be continuously extended to $\ell = 0$ (the case we excluded by the assumption $\gp \notin \mnf$). Once we set $\ell = 0$, this minimal angle is still a function of~$q$ and~$\chi$, and its minimum is a lower bound for the angle for any~$\ell$.
				
				Calculating this minimum is very complicated however, so we again resort to a computer proof with \textit{Mathematica}, see file \texttt{AngleBetweenProjectionAndHalfline.nb} at\\ \url{https://people.math.ethz.ch/~skalisnik/AngleBetweenProjectionAndHalfline.nb}.
			\end{proof}
			
			The desired deformation retraction should flow in the direction of~$\lavf$. However, the field~$\lavf$ is defined only on a single ellipsoid~$\of$. Two such vector fields generally do not coincide on the intersection of two (or more) ellipsoids, so we use the partition of unity, constructed in Subsection~\ref{subsection:partition-of-unity}, to merge the vector fields~$\lavf$ into one.
			
			Define the vector field $\avf\colon \on \to \RR^\ad$ as
			\[\avf(\gp) \dfeq \pun(\gp) \;\! \prv(\gp) + \sum_{\smpp \in \smp} \pu(\gp) \;\! \lavf(\gp).\]
			We understand this definition in the usual sense: this sum has only finitely many non-zero terms at each~$\gp$ (in fact at most two by Proposition~\ref{proposition:supports-pairwise-disjoint}), and outside of the ellipsoid~$\of$, we take the value of~$\lavf$ to be~$0$.
			
			\begin{corollary}\label{corollary:normal-projection-away-from-normal-space-of-ellipsoid}
				\
				\begin{enumerate}
					\item
						If $\smpp \in \smp$ and $\gp \in \of$, then
						\[\spr{\lavf(\gp)}{\prv(\gp)} = \Big(\big\|\prv(\gp)\big\| \cdot \cos(\pa)\Big)^2 \geq \tfrac{2}{3} \;\! \big\|\prv(\gp)\big\|^2.\]
					\item
						If $\gp \in \on$, then $\spr{\avf(\gp)}{\prv(\gp)} \geq \tfrac{2}{3} \;\! \big\|\prv(\gp)\big\|^2$.
				\end{enumerate}
				In particular, these two scalar products are non-zero outside~$\mnf$. Hence the fields $\lavf$ and $\avf$ have no zeros outside~$\mnf$.
			\end{corollary}
			
			\begin{proof}
				\
				\begin{enumerate}
					\item
						Assume first that $\prv(\gp)$ points into the half-space bounded by~$\ts[\bd{\cf{\smpp}{\pep(\gp)}}]{\gp}$ which contains~$\cf{\smpp}{\pep(\gp)}$. Then $\lavf(\gp) = \prv(\gp)$ and $\pa = 0$, so the statement is clear.
						
						Otherwise, $\lavf(\gp)$ is the orthogonal projection of $\prv(\gp)$ onto~$\ts[\bd{\cf{\smpp}{q}}]{\gp}$, so ${\big\|\lavf(\gp)\big\| = \big\|\prv(\gp)\big\| \cdot \cos(\pa)}$ and
						\[\spr{\lavf(\gp)}{\prv(\gp)} = \big\|\lavf(\gp)\big\| \cdot \big\|\prv(\gp)\big\| \cdot \cos(\pa) = \Big(\big\|\prv(\gp)\big\| \cdot \cos(\pa)\Big)^2.\]
						
						For the inequality, we use Lemma~\ref{lemma:normal-projection-away-from-normal-space-of-ellipsoid} to get $\cos^2(\pa) \geq \frac{2}{3}$.
					\item
						We have
						\begin{align*}
							\spr{\avf(\gp)}{\prv(\gp)} &= \spr{\pun(\gp) \;\! \prv(\gp) + \sum_{\smpp \in \smp} \pu(\gp) \;\! \lavf(\gp)}{\prv(\gp)} \\
							&= \pun(\gp) \big\|\prv(\gp)\big\|^2 + \sum_{\smpp \in \smp} \pu(\gp) \;\! \spr{\lavf(\gp)}{\prv(\gp)} \\
							&\geq \pun(\gp) \big\|\prv(\gp)\big\|^2 + \sum_{\smpp \in \smp} \pu(\gp) \;\! \tfrac{2}{3} \;\! \big\|\prv(\gp)\big\|^2 \\
							&\geq \tfrac{2}{3} \;\! \big\|\prv(\gp)\big\|^2.
						\end{align*}
				\end{enumerate}
			\end{proof}
			
			There is one more problem with taking~$\avf$ as the direction vector field of the deformation retraction. The closer $\gp$ is to the manifold, the shorter the vector $\prv(\gp)$, and thus~$\avf(\gp)$, is. If we used~$\avf$ as the velocity vector field for the flow, we would need infinite time to reach the manifold~$\mnf$. If we scale the vector field in the way that the distance to the manifold decreases with speed~$1$, we are sure to reach the manifold within time~$1$ which is how one usually gives a deformation retraction (or more generally any homotopy).
			
			Since $\dst(\gp, \mnf) = \dst(\gp, \pr(\gp))$, we need to divide~$\avf(\gp)$ with the length of its projection onto the vector~$\prv(\gp)$. Hence the following definition of the vector field $\vf\colon \on \setminus \mnf \to \RR^\ad$:
			\[\vf(\gp) \dfeq \frac{\|\prv(\gp)\|}{\spr{\avf(\gp)}{\prv(\gp)}} \;\! \avf(\gp).\]
			Corollary~\ref{corollary:normal-projection-away-from-normal-space-of-ellipsoid} ensures that the vector field~$\vf$ is well defined and that it has the same direction as~$\avf$.
			
			\begin{proposition}\label{proposition:vector-fields-bounded-and-lipschitz}
				For every $\smpp \in \smp$ the field $\lavf\colon \of \to \RR^\ad$ is bounded Lipschitz. The fields $\avf\colon \on \to \RR^\ad$ and ${\vf\colon \on \setminus \mnf \to \RR^\ad}$ are bounded and locally Lipschitz.
			\end{proposition}
			
			\begin{proof}
				The projection onto a half-space is \mbox{$1$-Lipschitz}; since the map $\prv$ is bounded in norm (by Lemma~\ref{lemma:distance-to-manifold} we have $\|\prv(\gp)\| = \dst(\gp, \pr(\gp)) = \dst(\gp, \mnf) \leq \pep(\gp) < \pp$), the field~$\lavf$ is also bounded. Lemma~\ref{lemma:continuity-of-closest-point-projection} tells us that the map~$\prv$ is \mbox{$(\frac{1}{1-\pp} + 1)$-Lipschitz} on $\of \subseteq \ctnm{1-\pp}$. As the composition of two Lipschitz maps, the vector field~$\lavf$ is Lipschitz with the product Lipschitz coefficient, i.e.~also~$\frac{1}{1-\pp} + 1$.
				
				Since the norm of the map~$\prv$, as well as all~$\lavf$, has the same bound~$\pp$, this is also a bound on the norm of~$\avf$:
				\[\big\|\avf(\gp)\big\| = \big\|\pun(\gp) \;\! \prv(\gp) + \sum_{\smpp \in \smp} \pu(\gp) \;\! \lavf(\gp)\big\| \leq\]
				\[\leq \pun(\gp) \;\! \big\|\prv(\gp)\big\| + \sum_{\smpp \in \smp} \pu(\gp) \;\! \big\|\lavf(\gp)\big\| \leq \Big(\pun(\gp) + \sum_{\smpp \in \smp} \pu(\gp)\Big) \;\! \pp = \pp.\]
				The field~$\avf$ is locally Lipschitz by Corollary~\ref{corollary:partition-of-unity-locally-lipschitz-amalgamation-is-locally-lipschitz}.
				
				Assume now that $\gp \in \on \setminus \mnf$. Recall from Lemma~\ref{lemma:normal-projection-away-from-normal-space-of-ellipsoid} that $\cos(\pa) \geq \sqrt{\frac{2}{3}}$. Thus
				\begin{align*}
					\frac{\|\prv(\gp)\|}{\spr{\avf(\gp)}{\prv(\gp)}} &= \frac{1}{\spr{\pun(\gp) \;\! \prv(\gp) + \sum_{\smpp \in \smp} \pu(\gp) \;\! \lavf(\gp)}{\frac{\prv(\gp)}{\|\prv(\gp)\|}}} \\
					&= \frac{1}{\pun(\gp) \|\prv(\gp)\| + \sum_{\smpp \in \smp} \pu(\gp) \;\! \spr{\lavf(\gp)}{\frac{\prv(\gp)}{\|\prv(\gp)\|}}} \\
					&= \frac{1}{\pun(\gp) \|\prv(\gp)\| + \sum_{\smpp \in \smp} \pu(\gp) \;\! \|\lavf(\gp)\| \cos(\pa)} \\
					&\leq \frac{1}{\pun(\gp) \|\prv(\gp)\| + \sqrt{\frac{2}{3}} \sum_{\smpp \in \smp} \pu(\gp) \;\! \|\lavf(\gp)\|} \\
					&\leq \frac{1}{\sqrt{\frac{2}{3}} \big\|\pun(\gp) \;\! \prv(\gp) + \sum_{\smpp \in \smp} \pu(\gp) \;\! \lavf(\gp)\big\|} \\
					&= \frac{1}{\sqrt{\frac{2}{3}} \|\avf(\gp)\|}.
				\end{align*}
				It follows that the norm of~$\vf$ is bounded by~$\sqrt{\frac{3}{2}}$.
				
				Let $\st{U}$ be a neighbourhood of~$\gp$, where $\avf$ is Lipschitz. Let $r \in \RR_{> 0}$ be such that $\ob{\gp}{r} \subseteq \st{U}$ and $r < \dst(\mnf, \gp) = \|\prv(\gp)\|$ and $r < 1 - \dst(\mnf, \gp)$. We claim that $\vf$ is Lipschitz on $\ob{\gp}{r}$ and therefore locally Lipschitz.
				
				By Lemma~\ref{lemma:continuity-of-closest-point-projection} the map~$\prv$ is Lipschitz on ${\ob{\gp}{r} \subseteq \otnm{1 - \dst(\mnf, \gp)}}$. The map~$\|\prv(\ph)\|$ is a composition of Lipschitz maps and therefore Lipschitz on~$\ob{\gp}{r}$. Clearly, it is also bounded.
				
				Since $\avf$ is also bounded and Lipschitz on ${\ob{\gp}{r} \subseteq \st{U}}$, so is the scalar product ${\sgp \mapsto \spr{\avf(\sgp)}{\prv(\sgp)}}$ by Lemma~\ref{lemma:lipschitz}. Recall from Corollary~\ref{corollary:normal-projection-away-from-normal-space-of-ellipsoid} that
				\[\spr{\avf(\sgp)}{\prv(\sgp)} \geq \tfrac{2}{3} \big\|\prv(\sgp)\big\| > \tfrac{2}{3} \big(\dst(\mnf, \gp) - r\big) > 0.\]
				Hence Lemma~\ref{lemma:lipschitz} also tells us that the map $\sgp \mapsto \frac{\|\prv(\sgp)\|}{\spr{\avf(\sgp)}{\prv(\sgp)}}$ is bounded Lipschitz on~$\ob{\gp}{r}$, and then so is its product with~$\avf$, i.e.~the field~$\vf$.
			\end{proof}
			
			The reason we consider the local Lipschitz property is that it allow us to define the flow of the field~$\vf$.
		
		\subsection{The Flow of the Vector Field}\label{subsection:flow}
		
			We will use the flow of the vector field~$\vf$ as part of the definition of the desired deformation retraction. Generally the flow of a vector field need not exist globally, and in our case the whole point is that the flow takes us to the manifold where the vector field is not defined. However, before we can establish what the exact domain of definition for the flow is, we will already need to refer to the flow to prove some of its properties. As such, it will be convenient to treat the flow as a partial function. Also, it is convenient to use \df{Kleene equality}~$\ke$ in the context of partial functions: $a \ke b$ means that $a$ is defined if and only if $b$ is, and is they are defined, they are equal.
			
			The flow of the vector field~$\vf\colon \on \setminus \mnf \to \RR^\ad$ can thus be given as a partial map ${\flow\colon (\on \setminus \mnf) \times \RR_{\geq 0} \parto \on \setminus \mnf}$ which satisfies the following for all $\gp \in \on \setminus \mnf$ and $t, u \in \RR_{\geq 0}$:
			\begin{enumerate}
				\item
					the domain of definition of~$\flow$ is an open subset of~$(\on \setminus \mnf) \times \RR_{\geq 0}$,
				\item
					the flow~$\flow$ is continuous everywhere on its domain of definition,
				\item
					if $\flow(\gp, u)$ is defined and $t \leq u$, then $\flow(\gp, t)$ is defined,
				\item
					$\flow(\gp, 0) \ke \gp$,
				\item
					$\flow\big(\flow(\gp, t), u\big) \ke \flow(\gp, t + u)$,
				\item
					if $\flow(\gp, u)$ is defined, the derivative of the function $\flow(\gp, \ph)$ exists at~$u$, and is equal to~$\vf\big(\flow(\gp, u)\big)$.
			\end{enumerate}
			
			A standard result~\cite{coleman2012calculus} tells us that if a vector field is locally Lipschitz, it has a local vector flow. That is, for every $\gp \in \on \setminus \mnf$ there exists $\epsilon \in \RR_{> 0}$ such that $\flow(\gp, t)$ is defined for all $t \in \intco{0}{\epsilon}$.
			
			We claim that if we move with the flow~$\flow$ of the vector field~$\vf$, we approach the manifold~$\mnf$ with constant speed.
			
			\begin{lemma}\label{lemma:flow-distance-to-manifold}
				If $(\gp, u) \in (\on \setminus \mnf) \times \RR_{\geq 0}$ is in the domain of definition of~$\flow$, then
				\[\dst\big(\mnf, \flow(\gp, u)\big) = \dst(\mnf, \gp) - u.\]
			\end{lemma}
			
			\begin{proof}
				Consider the functions $\intcc{0}{u} \to \RR$, given by $t \mapsto \dst\big(\mnf, \flow(\gp, t)\big)$ and $t \mapsto \dst(\mnf, \gp) - t$. To show that these two functions are the same (and thus in particular coincide for $t = u$), it suffices to show that they match in one point and have the same derivative.
				
				For $t = 0$, we have $\dst\big(\mnf, \flow(\gp, 0)\big) = \dst(\mnf, \gp)$. The derivative of the second function is constantly~$-1$. We calculate the derivative of the first function via the chain rule. Take $t \in \intcc{0}{u}$ and introduce an orthonormal $\ad$-dimensional coordinate system with the origin in~$\sgp \dfeq \flow(\gp, t)$, such that the first coordinate axis points in the direction of~$\prv(\sgp)$. In this coordinate system, the Jacobian matrix of the map $\dst(\mnf, \ph)$ at~$\sgp$ is a matrix row with the first entry~$-1$ and the rest~$0$. We need to multiply this matrix with the column, the first entry of which is $\spr{\flow'(\gp, t)}{\tfrac{\prv(\sgp)}{\|\prv(\sgp)\|}}$, i.e.~the scalar projection onto the direction $\frac{\prv(\sgp)}{\|\prv(\sgp)\|}$ of the derivative of~$\flow(\gp, \ph)$ at~$\sgp$.
				
				By the chain rule, the derivative of the function $t \mapsto \dst\big(\mnf, \flow(\gp, t)\big)$ is therefore
				\[(-1) \cdot \spr{\flow'(\gp, t)}{\tfrac{\prv(\sgp)}{\|\prv(\sgp)\|}} = -\spr{\vf(\sgp)}{\tfrac{\prv(\sgp)}{\|\prv(\sgp)\|}} = -\spr{\tfrac{\|\prv(\sgp)\|}{\spr{\avf(\sgp)}{\prv(\sgp)}} \;\! \avf(\sgp)}{\tfrac{\prv(\sgp)}{\|\prv(\sgp)\|}} = -1,\]
				as required.
			\end{proof}
			
			The next lemma is a tool which serves as a form of induction for real intervals.
			\begin{lemma}\label{lemma:interval-induction}
				Let $a \in \RR_{\geq 0}$ and let $I$ be either the interval $\intco{0}{a}$ or the interval~$\intcc{0}{a}$. Let $L \subseteq I$ have the following properties:
				\begin{itemize}
					\item
						$L$ is a lower subset of~$I$ (i.e.~$\all{t, u \in I}{u \in L \land t \leq u \impl t \in L}$),
					\item
						$0 \in L$,
					\item
						for every $t \in L_{< a}$ there exists $u \in I_{> t}$ such that $u \in L$,
					\item
						for every $t \in I$, if $\intco{0}{t} \subseteq L$, then $t \in L$.
				\end{itemize}
				Then $L = I$.
			\end{lemma}
			
			\begin{proof}
				To prove $L = I$, it suffices to show that $L$ is non-empty, open and closed in~$I$ since $I$ is connected.
				
				Because $L$ contains~$0$, it is non-empty. Since $L$ is a lower subset of~$I$, the third assumption on~$L$ is equivalent to openness of~$L$, and the fourth assumption is equivalent to closedness of~$L$.
			\end{proof}
			
			\begin{lemma}\label{lemma:field-normal-direction}
				An $\epsilon \in \intoo{0}{\pp}$ exists so that for every $\gp \in \npa \setminus \cn{\pp-\epsilon}$ and every $\smpp \in \smp$, such that $\gp, \pr(\gp) \in \of$, the vector~$\vf(\gp)$ has the same direction as~$\prv(\gp)$.
			\end{lemma}
			
			\begin{proof}
				First we will require that $\epsilon < \frac{\pp-\lppb}{2}$. In that case the inequality $\dst(\gp, \pr(\gp)) < \frac{\pp-\lppb}{2}$ leads to contradiction $\gp \in \otnm{\frac{\pp-\lppb}{2}} \subseteq \on{\frac{\pp+\lppb}{2}} \subseteq \cn{\pp-\epsilon}$. Thus $\pr(\gp) \notin \ob{\gp}{\frac{\pp-\lppb}{2}}$.
				
				Consider the intersection of $\of$ with the closed half-space, bounded by the hyperplane, tangent to~$\bd\cf{\smpp}{\pep(\gp)}$, on the side not containing~$\of{\smpp}{\pep(\gp)}$. This intersection contains~$\gp$. If we take~$\gp$ arbitrarily close to~$\bd\of$ (i.e.~we consider $\pep(\gp)$ tending towards~$\pp$), the intersection is contained in arbitrarily small balls around~$\gp$. More explicitly, calculation shows that the intersection is contained in~$\ob{\gp}{\sqrt{\frac{(1+\pp)(\pp^2 - (\pep(\gp))^2)}{p}}}$.
				
				If we choose~$\epsilon$ so that this intersection is contained in~$\ob{\gp}{\frac{\pp-\lppb}{2}}$, then this intersection cannot contain~$\pr(\gp)$, whence $\lavf(\gp) = \prv(\gp)$. From $0 < \lppb < \pp < 1$ we get that
				\[\frac{\pp (3\pp^2 - \lppb^2 + 2 \pp (2+\lppb))}{1+\pp} > 0 \qquad \text{and} \qquad \pp - \frac{1}{2} \sqrt{\frac{\pp (3\pp^2 - \lppb^2 + 2 \pp (2+\lppb))}{1+\pp}} > 0.\]
				If we pick any $\epsilon \in \intoo{0}{\frac{\pp-\lppb}{2}}$ satisfying $\epsilon < \pp - \frac{1}{2} \sqrt{\frac{\pp (3\pp^2 - \lppb^2 + 2 \pp (2+\lppb))}{1+\pp}}$, then the aforementioned intersection is indeed contained in~$\ob{\gp}{\frac{\pp-\lppb}{2}}$.
				
				By assumption $\gp \in \npa$, i.e.~$\gp$ and~$\pr(\gp)$ are in the same open ellipsoid. Recall from Proposition~\ref{proposition:supports-pairwise-disjoint} that, aside from~$\pun$, at most one~$\pu$ is non-zero. Thus the vector~$\avf(\gp)$ is a convex combination of~$\lavf(\gp)$ and~$\prv(\gp)$, so it is equal to~$\prv(\gp)$. Hence $\vf(\gp)$ has the same direction as~$\prv(\gp)$ for our choice of~$\epsilon$.
			\end{proof}
			
%
			
			\begin{lemma}\label{lemma:flow-contained}
				An $\epsilon \in \intoo{0}{\pp}$ exists so that for every $\gp \in \cn{\pp-\epsilon}$ and every $u \in \RR_{\geq 0}$, for which $\flow(\gp, u)$ is defined, we have $\flow(\gp, u) \in \cn{\pp-\epsilon}$.
			\end{lemma}
			
			\begin{proof}
				Take any positive $\epsilon$, smaller than the one in Lemma~\ref{lemma:field-normal-direction}. Let $\gp \in \cn{\pp-\epsilon}$ and $u \in \RR_{\geq 0}$, so that $\flow(\gp, u)$ is defined. Let
				\[L \dfeq \set{t \in \intcc{0}{u}}{\flow(\gp, t) \in \cn{\pp-\epsilon}}.\]
				We use Lemma~\ref{lemma:interval-induction} to show $L = \intcc{0}{u}$; this finishes the proof.
				
				Clearly $0 \in L$ and $L$ is a lower set. It is the preimage of~$\cn{\pp-\epsilon}$ under the map ${\flow(\gp, \ph)\colon \intcc{0}{u} \to \RR^\ad}$, so it is closed. We only still need to see that for every $t \in L_{< u}$ there exists $t' \in \intoc{t}{u}$ such that $t' \in L$.
				
				If $\flow(\gp, t) \in \on{\pp-\epsilon}$, the requisite~$t'$ clearly exists. Assume now that $\flow(\gp, t) \in \bd\cn{\pp-\epsilon}$.
				
				Recall from Lemma~\ref{lemma:open-cover-of-ellipsoid-area} that $\sea \cup \npa = \on$. Suppose first that $\flow(\gp, t) \in \sea$. Then $\vf$ has the same direction as~$\lavf$ on some neighbourhood of~$\flow(\gp, t)$, meaning that it points into the interior of or is at worst tangent to~$\cf{\smpp}{\pep(\flow(\gp, t))}$ on this neighbourhood, so the flow stays for awhile in~$\cn{\pp-\epsilon}$ which gives us the requisite~$t'$. On the other hand, if $\flow(\gp, t) \in \npa$, then we have~$t'$ by Lemma~\ref{lemma:field-normal-direction}.
			\end{proof}
			
			We are now ready to prove that the domain of definition of~$\flow$ is
			\[\fd \dfeq \set[1]{(\gp, t) \in (\on \setminus \mnf) \times \RR_{\geq 0}}{t < \dst(\mnf, \gp)}.\]
			
			\begin{proposition}
				The flow~$\flow$ is defined on~$\fd$.
			\end{proposition}
			
			\begin{proof}
				The flow is defined as long as it remains within the domain of~$\vf$, i.e.~$\on \setminus \mnf$. Take any $\gp \in \on \setminus \mnf$ and define
				\[L \dfeq \set{t \in \intco{0}{\dst(\mnf, \gp)}}{\text{$\flow$ is defined at $(\gp, t)$}}.\]
				We verify the properties for~$L$ from Lemma~\ref{lemma:interval-induction} to get $L = \intco{0}{\dst(\mnf, \gp)}$. The basic properties of the flow tell us that $0 \in L$ and that $L$ is an open lower subset of~$\intco{0}{\dst(\mnf, \gp)}$. Take $t \in \intco{0}{\dst(\mnf, \gp)}$ such that $\intco{0}{t} \subseteq L$. Because the vector field~$\vf$ is bounded (Proposition~\ref{proposition:vector-fields-bounded-and-lipschitz}), the map $\flow(\gp, \ph)\colon \intco{0}{t} \to \on \setminus \mnf$ (of which the field is the derivative) is Lipschitz, in particular uniformly continuous. Hence it has a (uniformly) continuous extension $\intcc{0}{t} \to \cn$ (since $\cn$, as a closed subspace of~$\RR^\ad$, is complete). Thus the limit $\sgp \dfeq \lim_{t' \nearrow t} \flow(\gp, t')$ exists and is in~$\cn$.
				
				We need to show that $\sgp \in \on \setminus \mnf$. Using Lemma~\ref{lemma:flow-distance-to-manifold}, we get
				\[\dst(\mnf, \sgp) = \dst\big(\mnf, \lim_{t' \nearrow t} \flow(\gp, t')\big) = \lim_{t' \nearrow t} \dst\big(\mnf, \flow(\gp, t')\big) = \lim_{t' \nearrow t} \big(\dst(\mnf, \gp) - t'\big) = \dst(\mnf, \gp) - t > 0.\]
				Hence $\sgp \notin \mnf$.
				
				We also have $\sgp \in \on$. Before the flow could leave~$\on$, it would have to get arbitrarily close to~$\bd\cn$ which would contradict Lemma~\ref{lemma:flow-contained}.
			\end{proof}
		
		\subsection{The Deformation Retraction}\label{subsection:deformation-retraction}
		
			We can now define a deformation retraction from~$\on$ to~$\mnf$. The flow~$\flow$ takes us arbitrarily close to the manifold without actually reaching it, so we will define the deformation retraction in two parts: first from $\on$ to a small neighbourhood of~$\mnf$, and then from this neighbourhood to~$\mnf$ itself.
			
			Recall that by assumption $\pp > \lppb$, and we have
			\[\mnf \subseteq \bigcup_{\smpp \in \smp} \cf{\smpp}{\lppb} \subseteq \bigcup_{\smpp \in \smp} \of.\]
			The distance of a point in~$\cf{\smpp}{\lppb}$ to the complement of~$\of$ is the smallest in a co-vertex of~$\cf{\smpp}{\lppb}$, where it is equal to~$\pp - \lppb$. Hence $\ctnm{\frac{\pp - \lppb}{2}} \subseteq \otnm{\pp - \lppb} \subseteq \on$.
			
			Denote $w \dfeq \frac{\pp - \lppb}{2}$ and define the map $\dr\colon \on \times \intcc{0}{1} \to \on$ by
			\[\dr(\gp, t) \dfeq \begin{cases} \flow\big(\gp, \min\set{\dst(\mnf, \gp) - w, t}\big) & \text{if $\gp \in \on \setminus \otnm{w}$,} \\ \gp & \text{if $\gp \in \ctnm{w}$.} \end{cases}\]
			This map is well defined: if $\dst(\mnf, \gp) = w$, the two function rules match. Each of them is continuous and defined on a domain, closed in~$\on \times \intcc{0}{1}$, so $R$ is continuous on~$\on \times \intcc{0}{1}$. Clearly $\dr$ is a strong deformation retraction from~$\on$ to~$\ctnm{w}$: for $\gp \in \on \setminus \otnm{w}$ we have
			\[\dr(\gp, 1) = \flow\big(\gp, \min\set{\dst(\mnf, \gp) - w, 1}\big) = \flow\big(\gp, \dst(\mnf, \gp) - w\big),\]
			so $\dst\big(\mnf, \dr(\gp, 1)\big) = \dst(\mnf, \gp) - \big(\dst(\mnf, \gp) - w\big) = w$ by Lemma~\ref{lemma:flow-distance-to-manifold}.
			
			\begin{proposition}\label{proposition:existence-of-deformation-retraction}
				There exists a strong deformation retraction of~$\on$ to~$\mnf$.
			\end{proposition}
			
			\begin{proof}
				First use~$\dr$ to strongly deformation retract $\on$ to~$\ctnm{w}$. From here, the usual normal deformation retraction works. Specifically, since $w$ is less that the reach of~$\mnf$, the map~$\pr$ is defined on~$\ctnm{w}$. Hence the map $\ctnm{w} \times \intcc{0}{1} \to \ctnm{w}$, given by $(\gp, t) \mapsto (1-t) \cdot \gp + t \cdot \pr(\gp)$, is well defined and a strong deformation retraction from~$\ctnm{w}$ to~$\mnf$.
			\end{proof}

	\section{Main Theorem}\label{section:main-result}
	
		\begin{theorem}\label{theorem:main}
			Let $\ad \in \NN$ and let $\mnf$ be a non-empty properly embedded $\C{1}$-submanifold of~$\RR^\ad$ without boundary. Let $\mnf$ have the same dimension~$\md$ around every point. Let $\smp \subseteq \mnf$ be a subset of~$\mnf$, locally finite in~$\RR^\ad$ (the sample from the manifold~$\mnf$). Let $\rch$ be the reach of~$\mnf$ in~$\RR^\ad$ and $\hd$ the Hausdorff distance between $\smp$ and~$\mnf$. Following the notation from Section~\ref{section:program}, let $m_\pp = 0.5$, $M_\pp = 0.96$, $\hd_\mathrm{off} = 0.55$. Then for all $\pp \in \intcc{m_\pp \rch}{M_\pp \rch}$ which satisfy
			\[\hd < \sqrt{2\pp \big(\sqrt{\rch (\pp + 2\rch)} - \rch\big) - \hd_\mathrm{off} \rch^2}\]
			there exists a strong deformation retraction from~$\on$ (the union of open ellipsoids around sample points) to~$\mnf$. In particular, $\mnf$, $\on$ and the nerve complex of the ellipsoid cover $\big(\of\big)_{\smpp \in \smp}$ are homotopy equivalent, and so have the same homology.
		\end{theorem}
		
		\begin{proof}
			First consider the case $\rch = \infty$. Then $\mnf$ is an $\md$-dimensional affine subspace of~$\RR^\ad$. In that case $\on$ is just~$\otnm{\pp}$ (the tubular neighbourhood around~$\mnf$ of radius~$\pp$) which clearly strongly deformation retracts to~$\mnf$ via the normal deformation retraction.
			
			A particular case of this is when $\md = \ad$ or when $\mnf$ is a single point. If $\md = 0$, i.e.~$\mnf$ is a non-empty locally finite discrete set of points, and if $\mnf$ has at least two points, then the reach is half the distance between two closest points. In this case we necessarily have $\smp = \mnf$ and the set~$\on$ is a union of $\pp$-balls around points in~$\mnf$ which clearly deformation retracts to~$\mnf$.
			
			We now consider the case $\rch < \infty$ and $0 < \md < \ad$.
			
			All our conditions and results are homogeneous in the sense that they are preserved under uniform scaling. In particular, we may rescale the whole space~$\RR^\ad$ by the factor~$\frac{1}{\rch}$ and may thus without loss of generality assume $\rch = 1$. The result now follows from Proposition~\ref{proposition:existence-of-deformation-retraction}.
		\end{proof}
		
		\begin{corollary}\label{corollary:sample-density}
			Let $\ad \in \NN$ and let $\mnf$ be a non-empty properly embedded $\C{1}$-submanifold of~$\RR^\ad$ without boundary. Let $\mnf$ have the same dimension~$\md$ around every point. Let $\smp \subseteq \mnf$ be a subset of~$\mnf$, locally finite in~$\RR^\ad$ (the sample from the manifold~$\mnf$). Let $\rch$ be the reach of~$\mnf$ in~$\RR^\ad$ and $\hd$ the Hausdorff distance between $\smp$ and~$\mnf$. Then whenever
			\[\frac{\hd}{\rch} < \sqrt{2 M_\pp (\sqrt{2 + M_\pp} - 1) - \hd_\mathrm{off}} \approx 0.913,\]
			there exists $\pp \in \RR_{> 0}$ such that $\mnf$ is homotopy equivalent to~$\cn$.
		\end{corollary}
		
		\begin{proof}
			The expression $\sqrt{2\pp \big(\sqrt{\rch (\pp + 2\rch)} - \rch\big) - \hd_\mathrm{off} \rch^2}$ is increasing in~$\pp$. Hence we get the required result from Theorem~\ref{theorem:main} by taking $\pp = M_\pp \rch$.
		\end{proof}
	
	\section{Discussion}\label{section:discussion}
	
		As already mentioned in the introduction, the ratio~$\frac{\hd}{\rch}$ is a measure of the density of the sample. We want the required sample density to be small, i.e.~the ratio~$\frac{\hd}{\rch}$ should be as large as possible. Corollary~\ref{corollary:sample-density} gave us the upper bound $\frac{\hd}{\rch} < 0.913$. For comparison, recall that Niyogi, Smale and~Weinberger~\cite{NSW} obtained the bound $\frac{\hd}{\rch} < \frac{1}{2} \sqrt{\frac{3}{5}} \approx 0.387$. Hence our result allows approximately $2.36$-times lower density of the sample.
		
		There is clear room for improvement of our result. The bounds we obtained from the theoretical parts of the proof yield
		\[\frac{\hd}{\rch} < \sqrt{2(\sqrt{3} - 1)} \approx 1.21,\]
		which would be a further improvement of the above result~$\frac{\hd}{\rch} < 0.913$ by around a third (more than three times an improvement over the Niyogi, Smale and~Weinberger's result). The only reason we had to settle for the worse result was because to prove Lemma~\ref{lemma:program-conclusion} in Section~\ref{section:program}, we used a computer program. We can get closer to the theoretical bound by increasing~$M_\pp$ and decreasing $m_\pp$ and~$\hd_\mathrm{off}$, in which case the program yields a smaller lower bound on the values of the calculated function. Hence we would need to run the program with smaller~$\delta$, but since the loops in the program, the number of steps in which is inversely proportional to~$\delta$, are nested four levels deep, dividing $\delta$ by some~$t$ makes the program run approximately $t^4$-times longer. The parameters, given in Section~\ref{section:program}, are what we settled for in this paper in order for the program to complete the calculation in a reasonable amount of time --- the program which computes in parallel ran for around $2.7$~days on a \mbox{4-core} Intel i7-7500 processor. Our bound on the sample density can thus be improved by anyone with more patience and better hardware.
		
		One of the questions we do not answer in this paper concerns robustment of our results to noise. This is relevant because in practice, the reach and the tangent spaces are estimated from the sample, and are thus only approximately known. The sample points might also lie only in the vicinity of the manifold, not exactly on it. In this paper we wanted to establish the new methods, and we leave their refinement to take noise into account for future work.
		
		Our result is expressed in terms of the reach of a manifold, which is a global feature. As the classical sampling theory advanced, researchers refined the notion of reach to local feature size, weak feature size, $\mu$-reach and related concepts~\cite{amenta1999surface}, \cite{chazal2005weak}, \cite{chazal2008smooth}, \cite{Chazal2009}, \cite{Turner}, \cite{dey2017parameter}. A natural question arises whether we can apply these concepts to improve the bounds on a (local) density of a sample when using ellipsoids. In particular, it would be interesting to see whether we can improve our result by allowing differently sized ellipsoids around different sample points, with the upper bound on the size given in terms of the local feature size (local distance to the medial axis) or the distance to critical points.

\noindent
{\bf Acknowledgements.} We thank Paul Breiding, Peter Hintz and Jure Kali\v{s}nik for helpful discussions.
 \bigskip \medskip

\bibliography{ellipsoid} 
\bibliographystyle{plain}

\noindent
\footnotesize {\bf Authors' affiliations:}

\smallskip

\noindent
Sara Kali\v snik  \\  Department of Mathematics, ETH Zurich, Switzerland
\\ {\tt sara.kalisnik@math.ethz.ch}

\noindent Davorin Le\v{s}nik
 \\  Faculty of Mathematics and Physics, University of Ljubljana, Slovenia \\  {\tt davorin.lesnik@fmf.uni-lj.si}

\end{document}